\documentclass[11pt,reqno]{amsart}
\usepackage{graphicx,enumerate,amssymb,bbold}
\usepackage{sidecap}

\usepackage{floatrow}
\usepackage[notrig]{physics}
\usepackage[font=footnotesize,labelfont=small]{subcaption}
\usepackage[linesnumbered,ruled,vlined]{algorithm2e}
\usepackage[left=2.5cm,right=2.5cm,top=2.5cm,bottom=2.5cm,includeheadfoot]{geometry}
\usepackage[thinlines]{easytable}

\usepackage[export]{adjustbox}
\usepackage{scalerel}
\usepackage{xparse}
\usepackage{amsmath}
\usepackage{environ}
\usepackage[titletoc]{appendix}
\usepackage{xcolor}
\newcommand\crule[3][black]{\textcolor{#1}{\rule{#2}{#3}}}

\ExplSyntaxOn
\NewEnviron{bmatrixT}
 {
  \marine_transpose:V \BODY
 }

\int_new:N \l_marine_transpose_row_int
\int_new:N \l_marine_transpose_col_int
\seq_new:N \l_marine_transpose_rows_seq
\seq_new:N \l_marine_transpose_arow_seq
\prop_new:N \l_marine_transpose_matrix_prop
\tl_new:N \l_marine_transpose_last_tl
\tl_new:N \l_marine_transpose_body_tl

\cs_new_protected:Nn \marine_transpose:n
 {
  \seq_set_split:Nnn \l_marine_transpose_rows_seq { \\ } { #1 }
  \int_zero:N \l_marine_transpose_row_int
  \prop_clear:N \l_marine_transpose_matrix_prop
  \seq_map_inline:Nn \l_marine_transpose_rows_seq
   {
    \int_incr:N \l_marine_transpose_row_int
    \int_zero:N \l_marine_transpose_col_int
    \seq_set_split:Nnn \l_marine_transpose_arow_seq { & } { ##1 }
    \seq_map_inline:Nn \l_marine_transpose_arow_seq
     {
      \int_incr:N \l_marine_transpose_col_int
      \prop_put:Nxn \l_marine_transpose_matrix_prop
       {
        \int_to_arabic:n { \l_marine_transpose_row_int }
        ,
        \int_to_arabic:n { \l_marine_transpose_col_int }
       }
       { ####1 }
     }
   }
   \tl_clear:N \l_marine_transpose_body_tl
   \int_step_inline:nnnn { 1 } { 1 } { \l_marine_transpose_col_int }
    {
     \int_step_inline:nnnn { 1 } { 1 } { \l_marine_transpose_row_int }
      {
       \tl_put_right:Nx \l_marine_transpose_body_tl
        {
         \prop_item:Nn \l_marine_transpose_matrix_prop { ####1,##1 }
         \int_compare:nF { ####1 = \l_marine_transpose_row_int } { & }
        }
      }
     \tl_put_right:Nn \l_marine_transpose_body_tl { \\ }
    }
   \begin{bmatrix}
   \l_marine_transpose_body_tl
   \end{bmatrix}
 }
\cs_generate_variant:Nn \marine_transpose:n { V }
\cs_generate_variant:Nn \prop_put:Nnn { Nx }
\ExplSyntaxOff

\usepackage{tikz,tikz-cd}
\usetikzlibrary{calc}

\newcommand{\DrawBoxFwd}[1][]{%
    \tikz[overlay,remember picture]{
\fill[rectangle,rounded corners,fill=red!15,draw, fill opacity=0.5,thick,inner sep=0pt]
      ($(left)+(-0.8em,1.7em)$) rectangle
      ($(right)+(0.2em,-1.2em)$);}
}
\newcommand{\DrawBoxAdj}[1][]{%
    \tikz[overlay,remember picture]{
\fill[rectangle,rounded corners,fill=blue!15,draw, fill opacity=0.5,thick,inner sep=0pt]
      ($(left)+(-0.3em,1.7em)$) rectangle
      ($(right)+(0.2em,-1.2em)$);}
}

\usetikzlibrary{positioning}
\tikzstyle{mybox} = [draw=blue!50, fill=blue!20, very thick,
    rectangle, rounded corners, inner sep=10pt, inner ysep=15pt]
\tikzstyle{fancytitle} =[fill=gray!20, rounded corners, text=black]
\usepackage{tabularx}
\newcolumntype{?}{!{\vrule width 1pt}}
\usepackage{tabulary}
\usepackage{multirow}
\usepackage[bookmarks=false, colorlinks=true, pdfstartview=FitV, linkcolor=blue,
            citecolor=blue, urlcolor=blue]{hyperref}
\usepackage{makecell}
\usepackage[makeroom]{cancel}

\numberwithin{equation}{section}
\numberwithin{figure}{section}
\theoremstyle{plain}
\newtheorem{theorem}{Theorem}[section]
\newtheorem{lemma}[theorem]{Lemma}
\newtheorem{proposition}[theorem]{Proposition}

\theoremstyle{definition}

\newtheorem{remark}{Remark}[section]

\newcommand{\bitem}{\begin{itemize}}
\newcommand{\eitem}{\end{itemize}}
\newcommand{\mc}[1]{\mathcal{#1}}

\newcommand{\N}{\mathbb{N}}
\newcommand{\R}{\mathbb{R}}

\newcommand{\bpm}{\begin{pmatrix}}
\newcommand{\epm}{\end{pmatrix}}
\newcommand{\bsm}{\left(\begin{smallmatrix}}
\newcommand{\esm}{\end{smallmatrix}\right)}
\newcommand{\T}{\top}

\newcommand{\ol}[1]{\overline{#1}}
\newcommand{\la}{\langle}
\newcommand{\ra}{\rangle}

\newcommand{\gdw}{\Leftrightarrow}

\newcommand{\eins}{\mathbb{1}}

\DeclareMathOperator{\Diag}{Diag}

\DeclareMathOperator{\ggrad}{grad}

\DeclareMathOperator{\Exp}{Exp}

\DeclareMathOperator{\KL}{KL}

\newcommand{\BS}[1]{{\eins_{\mc{S}_{#1}}}}
\newcommand{\BW}{{\eins_{\mc{W}}}}
\newcommand{\BP}{{\eins_{\mc{P}}}}
\newcommand{\PT}{{\Pi}}
\newcommand{\RO}{{R}}
\newcommand{\Egrad}{\partial}
\newcommand{\Rgrad}{\ggrad}

\title[Learning Adaptive Regularization Using Geometric Assignment]{Learning Adaptive Regularization for Image Labeling\\ Using Geometric Assignment}

\author[R.~H\"{u}hnerbein, F.~Savarino, S.~Petra, C.~Schn\"{o}rr]{Ruben H\"{u}hnerbein, Fabrizio Savarino, Stefania Petra, Christoph Schn\"{o}rr}

\address[R.~H\"{u}hnerbein, F.~Savarino]{Image and Pattern Analysis Group, Heidelberg University, Germany}
\email{\{ruben.huehnerbein,fabrizio.savarino\}@iwr.uni-heidelberg.de}

\address[S.~Petra]{Mathematical Imaging Group, Heidelberg University, Germany}
\email{petra@math.uni-heidelberg.de}
\urladdr{\url{https://www.stpetra.com/}}

\address[C.~Schn\"{o}rr]{Image and Pattern Analysis Group, Heidelberg University, Germany}
\email{schnoerr@math.uni-heidelberg.de}
\urladdr{\url{http://ipa.math.uni-heidelberg.de}}

\subjclass[2010]{}

\begin{document}

\begin{abstract}
We study the inverse problem of model parameter learning for pixelwise image labeling, using the linear assignment flow and training data with ground truth. This is accomplished by a Riemannian gradient flow on the manifold of parameters that determines the regularization properties of the assignment flow. Using the symplectic partitioned Runge--Kutta method for numerical integration, it is shown that deriving the sensitivity conditions of the parameter learning problem and its discretization commute. A convenient property of our approach is that learning is based on exact inference. Carefully designed experiments demonstrate the performance of our approach, the expressiveness of the mathematical model as well as its limitations, from the viewpoint of statistical learning and optimal control.
\end{abstract}

\keywords{image labeling, assignment manifold, assignment flow, dynamical systems, replicator equation, evolutionary dynamics, sensitivity analysis,  parameter learning, adaptive regularization.}

\thanks{
Part of this research was performed while R.~H\"{u}hnerbein was visiting the Institute for Pure and Applied Mathematics (IPAM) at UCLA, which is supported by the National Science Foundation (Grant No. DMS-1440415).
Financial support by the German Science Foundation (DFG), grant GRK 1653, is gratefully acknowledged. This work has also been stimulated by the Heidelberg Excellence Cluster STRUCTURES, funded by the DFG under Germany's Excellence Strategy EXC-2181/1 - 390900948.}

\maketitle
\tableofcontents

\section{Introduction}
\label{sec:Introduction}

\subsection{Overview and Scope}

The \textit{image labeling problem}, i.e.~the problem to classify images pixelwise depending on the spatial context, has been thoroughly investigated during the last two decades {using discrete graphical models}. While the evaluation (inference) of such models is well understood \cite{Kappes:2015aa}, \textit{learning the parameters} of such models has remained elusive, in particular for models with higher connectivity of the underlying graph. Various sampling-based and other approximation methods exist (cf.~\cite{Zhu:2002aa} and references therein), but the \textit{relation} between approximations of the \textit{learning problem} on the one hand, and approximations of the subordinate \textit{inference problem} on the other hand, is less understood \cite{Wainwright:2006aa}.

In this paper, we focus on parameter learning for contextual pixelwise image labeling based on the \textit{assignment flow} introduced by \cite{Astroem2017}. In comparison with discrete graphical models, an antipodal viewpoint was adopted by \cite{Astroem2017} for the design of the assignment flow approach: Rather than performing \textit{non-smooth convex outer} relaxation and programming, followed by \textit{subsequent} rounding to integral solutions that is common when working with large-scale discrete graphical models, the assignment flow provides a \textit{smooth nonconvex interior} relaxation that performs rounding to integral solutions \textit{simultaneously}. Convergence and stability of the assignment flow has been studied in \cite{Zern:2020aa}, and extensions to unsupervised scenarios are reported in \cite{Zern:2020ab,Zisler:2019ab}. In \cite{Huhnerbein:2018aa} it was shown that the assignment flow can emulate a given discrete graphical model in terms of smoothed local Wasserstein distances, that evaluate the edge-based parameters of the graphical model. In comparison to established belief propagation iterations \cite{Yedidia:2005aa,Wainwright:2005aa}, the assignment flow driven by `Wasserstein messages' \cite{Huhnerbein:2018aa} continuously takes into account basic constraints, which enables to compute good suboptimal solutions just by numerically integrating the flow in a proper way \cite{Zeilmann:2018aa}. We refer to \cite{Schnorr:2019aa} for summarizing recent work based on the assignment flow and a discussion of further aspects.

In this paper, we ignore the connection to discrete graphical models and focus on the parameter learning problem for the assignment flow directly. This problem was raised in \cite[Section 5 and Fig.~14]{Astroem2017}. The present paper provides a detailed solution. Adopting the \textit{linear} assignment flow as introduced by \cite{Zeilmann:2018aa}, enables to cast the parameter estimation problem into the general form
\begin{subequations}\label{eq:parameter-problem-general}
\begin{align}
\min_{p \in \mc{P}}\quad \mc{C}\big( x&(T, p) \big ) \label{ivp-a}\\ \label{ivp-b}
\text{s.t.}\quad \dot x(t) &=f(x(t), p, t), \quad t \in [0,T], \\
\quad x(0) &= x_0,\label{ivp-c}
\end{align}
\end{subequations}
where the parameters $p$ determine the vector field of the linear assignment flow \eqref{ivp-b} whose unique solution is evaluated at some point of time $T$ by a suitable loss function \eqref{ivp-a}.  
This problem formulation has a range of advantages. 
\begin{itemize}
\item
\textit{Inference} (labeling) that always defines a subroutine of a learning procedure, can be carried out \textit{exactly} by means of numerically solving \eqref{ivp-b}. In other words, errors of approximate inference (e.g.~as they occur with graphical models) are absent and cannot compromise the effectiveness of parameter learning. 
\item
In addition, \textit{discretization effects} can be handled in the most convenient way: we show that the operations of (i) deriving the optimality conditions of \eqref{eq:parameter-problem-general} and (ii) problem discretization \textit{commute} if a proper numerical scheme is used. 
\end{itemize}
As a result, we obtain a well-defined and relatively simple algorithm for parameter learning that is easy to implement and enables reproducible research. We report the results of a careful numerical evaluation in order to highlight the scope of our approach and its limitations.

{We discuss in Section \ref{sec:Contribution-Organization} our specific contributions that elaborate a related conference paper \cite{Huhnerbein:2019aa} through the content of Section \ref{sec:Background} (sensitivity analysis, commutativity of diagram \ref{fig:sensitivity-diagram}, numerical schemes), Section \ref{sec:Adaptive-Learning-for-Dynamical-Systems} (parameter estimation, algorithm), Section \ref{sec:Experiments} (a range of experiments) and the Appendix (proofs).
}

\subsection{Related Work, Contribution and Organization}
\label{sec:Contribution-Organization}

The task to optimize parameters of a dynamical system \eqref{eq:parameter-problem-general} is a familiar one in the communities of scientific computing and optimal control \cite{tomovic1972general,cao_adjoint_2003}, but may be less known to the imaging community. Therefore, we provide the necessary background in Section \ref{sec:sensitivity-analysis}. 

Geometric numerical integration of ODEs on manifolds is a mature field as well \cite{Hairer:2006}. Here we have to distinguish between the integration of the assignment flow \cite{Zeilmann:2018aa} and integration schemes for numerically solving \eqref{eq:parameter-problem-general}. The task to design the latter schemes faces the `optimize-then-discretize' vs.~`discretize-then-optimize' dilemma. Conditions and ways to resolve this dilemma have been studied in the optimal control literature \cite{Hager:2000aa,Ross:2006aa}. See also the recent survey \cite{Sanz-Serna:2016aa} and references therein. We provide the corresponding background in Sections \ref{sec:Symplectiv-RK} and \ref{sec:Solving-Para-Prob} including a detailed proof of Theorem \ref{theorem:discrete-sensitivity} that is merely outlined in \cite{Sanz-Serna:2016aa}. 
The application to the linear assignment flow (Section \ref{sec:Assignment-Flow}) requires considerable work, taking into account that the state equation \eqref{ivp-b} derives from the full nonlinear geometric assignment flow (Section \ref{sec:Adaptive-Learning-for-Dynamical-Systems}). Section \ref{sec:Adaptive-Learning-for-Dynamical-Systems} concludes with specifying Algorithms \ref{alg:1} and \ref{alg:2} whose implementation realizes our approach.

From a more distant viewpoint, our work ties in with research on networks from a \textit{dynamical systems} point of view, that emanated from \cite{He:2016aa} in computer science and has also been promoted recently in mathematics \cite{E:2017aa}.  
The recent work \cite{Haber:2017aa}, for example, studied stability issues of discrete-time network dynamics using techniques of numerical ODE integration. The authors adopted the discretize-then-differentiate viewpoint on the parameter estimation problem and suggested symplectic numerical integration in order to achieve better stability. As mentioned above, our work contrasts in that inference is always \textit{exact}\footnote{{``Exact'' means that $T$ is chosen sufficiently large such that the assignment $W(T)$ is almost integral, i.e.~a labeling, according to the entropy-based criterion of \cite[Section 3.3.4]{Astroem2017}.}} during learning, unlike the more involved architecture of \cite{Haber:2017aa} where learning is based on \textit{approximate} inference. Furthermore, in our case, symplectic numerical integration is a \textit{consequence} of making the diagram of Figure \ref{fig:sensitivity-diagram} (page \pageref{fig:sensitivity-diagram}) \textit{commute}. This property   
qualifies our approach as a proper (though rudimentary) method of \textit{optimal control} (cf.~\cite{Ross:2006aa}).

{We numerically evaluate our approach in Section \ref{sec:Experiments} using three different experiments. The first experiment considers a scenario of 2 labels and images of binary letters. The results discussed in Section~\ref{sec:binary-letters} 
illustrate the \textit{adaptivity} of regularization by using \textit{non-uniform} weights that are predicted for \textit{novel unseen} image data.
The second experiment uses} a class of computer-generated random images such that learning the regularization parameters is \textit{necessary} for accurately labeling each image pixelwise. It is demonstrated in Section \ref{sec:curvilinear} that, for each given ground truth labeling, the parameter estimation problem can be solved \textit{exactly}. As a consequence, the performance of the assignment flow solely depends on the prediction map, i.e.~the ability to map features extracted from \textit{novel} data to proper weights as regularization parameters, using as examples both features and \textit{optimal} parameters computed during the training phase. For statistical reasons, this task becomes feasible if the domain of the prediction map is restricted to \textit{local} contexts, in terms of features observed within local windows. We discuss consequences for future work in Section \ref{sec:Conclusion}. 
Finally, in Section \ref{sec:label-transport}, we conduct an experiment that highlights the remarkable model expressiveness of the assignment flow as well as limitations that result from learning \textit{constant} parameters. 
 
We conclude in Section \ref{sec:Conclusion}.

\subsection{Basic Notation}
For the clarity of exposition we use general mathematical notation in Section \ref{sec:Background} that should be standard, whereas specific notation related to the assignment flow is introduced in Section \ref{sec:Assignment-Flow}. 

We set $[n]=\{1,2,\dotsc,n\}$ for $n \in \N$ and $\eins_{n} = (1,1,\dotsc,1)^{\T} \in \R^{n}$. For a matrix $A \in \R^{m\times n}$, the $i$-th row vector is denoted by $A_i,\, i \in [m]$ and its transpose by $A^\T \in \R^{n\times m}$. $\la a,b \ra$ denotes the Euclidean inner product of $a, b \in \R^{n}$ and $\la A, B\ra = \sum_{i\in[n]} \la A_i, B_i\ra$ the (Frobenius) inner product between two matrices $A, B \in \R^{m\times n}$. The probability simplex is denoted by $\Delta_{n} = \{p \in \R^{n} \colon p_{i} \geq 0,\, i \in [n],\, \la \eins_{n},p\ra=1\}$. Various orthogonal projections onto a convex set are generally denoted by $\Pi$ and distinguished by a corresponding subscript, like $\Pi_{n},\Pi_{\mc{P}}, \dotsb$, etc.

{The functions $\exp, \log$ apply \textit{componentwise} to strictly positive vectors $x \in \R_{++}^{n}$, e.g.~$e^{x}=(e^{x_{1}},\dotsc,e^{x_{n}})$, and similarly for strictly positive matrices. Likewise, if $x, y \in \R_{++}^{n}$, then we simply write 
\begin{equation}
x y = (x_{1} y_{1}, \dotsc, x_{n} y_{n}),\;
\frac{x}{y} = \Big(\frac{x_{1}}{y_{1}},\dotsc,\frac{x_{n}}{y_{n}}\Big)
\end{equation}
for the \textit{componentwise} multiplication and division.}

We assume the reader to be familiar with elementary notions of Riemannian geometry as found, e.g., in \cite{Lee:2013aa,Jost:2017aa}. Specifically, given a Riemannian manifold $(\mc{M},g)$ with metric $g$, and a smooth function $f \colon \mc{M} \to \R$, the Riemannian gradient of $f$ is denoted by $\ggrad f$ and given by
\begin{equation}\label{eq:def-R-gradient}
\la \ggrad f, X \ra_{g} = df(X),\qquad\forall X
\end{equation}
where $X$ denotes any smooth vector field on $\mc{M}$, that returns the tangent vector $X_{p} \in T_{p}\mc{M}$ when evaluated at $p \in \mc{M}$. The right-hand side of \eqref{eq:def-R-gradient} denotes the differential $df$ of $f$, acting on $X$. More generally, for a map $F \colon \mc{M} \to \mc{N}$ between manifolds, we write $dF(p)[v] \in T_{F(p)}\mc{N},\, p \in \mc{M},\,v \in T_{p}\mc{M}$, if the base point $p$ matters.

In the Euclidean case $f \colon \R^{n} \to \R$, the gradient is a column vector and denoted by $\partial f$. For $F \colon \R^{n} \to \R^{m}$, we identify the differential $dF \in \R^{m \times n}$ with the Jacobian matrix. If $x = (x_1, x_2)^\T \in \R^n = \R^{n_1}\times \R^{n_2}$ with $n = n_1 + n_2$, then the Jacobian of $F(x) = F(x_1, x_2)$ with respect to the parameter vector $x_i$ is denoted by $d_{x_i}F$, for $i =1, 2$.

\section{Sensitivity Analysis for Dynamical Systems}
\label{sec:Background}

In this section, we consider the constrained optimization problem \eqref{eq:parameter-problem-general}
with a smooth objective function $\mc{C} \colon \mathbb R^{n_x} \to \R$. The constraints are given by a general \textit{initial value problem (IVP)}, which consist of a system of ordinary differential equations (ODEs) \eqref{ivp-b} that is parametrized by a vector $p \in \mc{P} \subset \R^{n_p}$, and an initial value $x_0 \in \mathbb R^{n_x}$. To ensure existence, uniqueness and continuous differentiability of the solution trajectory $x(t)$ on the whole time horizon $[0, T]$, we assume that $f(\cdot,p,\cdot)$ of \eqref{ivp-b} is Lipschitz continuous on $\mathbb R^{n_x} \times [0, T]$, for any $p$. 
 
Since we assume the initial value $x_0$ and the time horizon $[0, T]$ to be fixed, the objective function \eqref{ivp-a}
\begin{equation}\label{eq:def-Phi}
  \Phi(p) := \mc{C}(x(T, p))
\end{equation}
effectively is a function of parameter $p$, i.e. $\Phi \colon \R^{n_p} \to \R$. In order to solve \eqref{eq:parameter-problem-general} with a gradient based method, we have to compute the \textit{gradient}
\begin{equation}\label{eq:adjoint-sensitivity-column-form}
\Egrad_p \Phi(p) = d_p x(T,p)^\T \Egrad_x \mc{C}(x(T,p)).
\end{equation}
The term $d_p x(T,p)$ -- called \textit{sensitivity} -- measures the sensitivity of the the solution trajectory $x(t)$ at time $T$ with respect to changes in the parameter $p$. Two basic approaches for determining \eqref{eq:adjoint-sensitivity-column-form} are stated in Section \ref{sec:sensitivity-analysis}, and we briefly highlight why using one of them, the \textit{adjoint approach}, is advantageous for computing sensitivities. In Section \ref{sec:Symplectiv-RK}, we recall symplectic Runge-Kutta methods and conditions for preserving quadratic invariants. The latter property relates to the derivation of a class of numerical methods such that  evaluating \eqref{eq:adjoint-sensitivity-column-form}, which derives from the time-continuous problem \eqref{eq:parameter-problem-general}, is \textit{identical} to first discretizing \eqref{eq:parameter-problem-general} followed by computing the corresponding derived expression \eqref{eq:adjoint-sensitivity-column-form}. Two specific instances of the general numerical scheme are detailed in Section \ref{sec:two-schemes}.

\subsection{Sensitivity Analysis}\label{sec:sensitivity-analysis}
In this section we describe how the sensitivity $d_p x(T,p)$ can be determined by solving one of the two initial value problems defined below: the \textit{variational system} and the \textit{adjoint system}.

\begin{theorem}[{\bfseries Variational System; \cite[Ch. I.14, Thm. 14.1]{Hairer:2008aa}}] \label{theorem-variational-system}
Suppose the derivatives $d_x f$ and $d_p f$ exist and are continuous in the neighborhood of the solution $x(t)$ for $t \in [0, T]$. Then, the sensitivity with respect to the parameters
\begin{equation}\label{eq:vde-p-T}
 d_{p} x(T, p) =: \delta(T)
\end{equation}
exists, is continuous and satisfies the \textit{variational system}
\begin{subequations}\label{eq:vde-p}
\begin{align}
\dot \delta(t) &=d_x f(x(t), p, t) \delta(t) + d_p f(x(t), p, t),
\label{vde-p-a}\\\label{vde-p-b}
\delta(0) &= 0 \in \R^{n_x \times n_p},
\end{align}
\end{subequations}
with $t \in [0, T]$ and $\delta(t) \in \R^{n_x \times n_p}$. If the initial value $x(0)$ \eqref{ivp-c} depends on the parameters $p$, the initial value \eqref{vde-p-b} has to be adjusted as $\delta(0) = d_p x(0)$.
\begin{proof} 
A detailed proof can be found in \cite[Ch. I.14, Thm. 14.1]{Hairer:2008aa}. In order to make this paper self-contained, a sketch of the argument follows.

The integral representation of the solution to \eqref{ivp-b} is given by $x(t,  p) = x_0 + \int_0^t f(x(s), p, s)ds$. Differentiating with respect to $p$ and exchanging integration and differentiation by the theorem of Lebesgue yields
\begin{subequations}
\begin{align}
&d_{p} x(t, p) = d_{p} x_0 + \int_{0}^{t} d_{p} \big ( f(x(s), p, s) \big )ds\\
 \begin{split}
&=d_{p} x_0 + \int_{0}^{t} \Big ( d_x f(x(s), p, s) d_{p} x(s, p) + d_p f(x(s), p, s) \Big ) \; ds. 
\end{split}
\end{align}
\end{subequations}
Substituting $\delta(t) = d_{p} x(t, p)$, gives
\begin{equation}
  \delta(t) = \delta_0 + \int_{0}^{t} d_x f(x(s), p, s) \delta(s) + d_p f(x(s), p, s) ds,
\end{equation}
which is the integral representation of the trajectory $\delta(t)$ solving \eqref{eq:vde-p}. \qed
\end{proof}
\end{theorem}

For the computation of the variational system \eqref{eq:vde-p} the solution $x(t)$ is required. Since the variational system \eqref{eq:vde-p} is a matrix-valued system of dimension
 $n_x \times n_p$, the size of the system grows with the number of parameters $n_p$. For small $n_{p}$, solving the variational system is efficient. In practice, it can be simultaneously integrated numerically together with system \eqref{ivp-b}.

\begin{theorem}[\bfseries Adjoint System]\label{theorem-adjoint-system}
Suppose that the derivatives $d_x f$ and $d_p f$ exist and are continuous in the neighborhood of the solution $x(t)$ for $t \in [0, T]$. Then, the sensitivity with respect to the 
parameters is given by
\begin{equation}\label{eq:ade-dp}
d_{p} x(T, p)^\T = \int_0^T d_p f(x(t),p,t)^{\T} \lambda(t) dt,
\end{equation}
where $\lambda(t)\in \R^{n_x \times n_x}$ solves the adjoint system
\begin{subequations}\label{eq:ade}
\begin{align}
\dot \lambda(t) & =-d_{x} f(x(t), p, t)^{\T} \lambda(t), \quad t \in [0, T],\label{ade-a}\\
\lambda(T) &= I \in \R^{n_x \times n_x}.\label{ade-b}
\end{align}
\end{subequations}

\begin{proof}
This proof is elaborated on in a broader context in Section \ref{sec:Solving-Para-Prob}.
\end{proof}
\end{theorem}

Similar to the variational system of Theorem \ref{theorem-variational-system}, solving the adjoint system \eqref{eq:ade} requires the solution $x(t)$. The adjoint system is matrix-
valued of dimension $n_x \times n_x$, in contrast to the variational system which is of dimension $n_x \times n_p$. Thus, if $n_p \gg n_x$ as will be the case in our scenario, it is more efficient to solve 
\eqref{eq:ade} instead of \eqref{eq:vde-p}. Another major difference is that the adjoint system is defined \textit{backwards} in time, starting from the endpoint $T$. This has 
important computational advantages for our setting. In view of the required gradient \eqref{eq:adjoint-sensitivity-column-form}, we are not interested in the full sensitivity but 
rather in the derivative along the direction $\eta:= \Egrad_x \mc{C}(x(T,p))$, i.e. $d_{p} x(T, p)^\T\eta$. This can be achieved by exploiting the structure of the adjoint system, by multiplying \eqref{eq:ade} from the right by $\eta$ and setting $\ol{\lambda}(t) := \lambda(t) \eta$. The resulting IVP is again an adjoint system, no longer being matrix-valued but vector-valued $\ol{\lambda}(t) \in \R^{n_x}$, with $\ol{\lambda}(T) = \eta \in \R^{n_x}$. Thus, from now on, we consider the latter case and denote $\ol{\lambda}(t)$ again by $\lambda(t)$, which is vector-valued.

As a consequence, we will focus on the adjoint system \eqref{eq:ade} in the remainder of this paper. In particular, \eqref{eq:ade-dp} will be used to estimate parameters $p$ by solving \eqref{eq:parameter-problem-general} using a gradient descent flow. This requires to solve the adjoint system numerically. However, a viable alternative to this `optimize-then-discretize' approach is to reverse this order, that is to discretize problem \eqref{eq:parameter-problem-general} first, and then to derive a corresponding \textit{time-discrete} adjoint system. It turns out that both ways are equivalent if a proper class of numerical integration scheme is chosen for discretizing the system in time. This will be shown in Section \ref{sec:Solving-Para-Prob} after collecting required background material in Section \ref{sec:Symplectiv-RK}. 

\subsection{Symplectic Partitioned Runge-Kutta Methods}\label{sec:Symplectiv-RK}
In this section, we recall basic concepts of numerical integration from \cite{Hairer:2006,Sanz-Serna:2016aa} in order to prepare Section \ref{sec:Solving-Para-Prob}. Symplectic schemes are typically applied to Hamiltonian systems in order to conserve certain quantites, often with a physical background. The pseudo-Hamiltonian defined below by \eqref{eq:def-pseudo-Hamiltonian} will play a similar role, albeit there is no physical background for our concrete scenario to be studied in subsequent Sections.

A general \textit{$s$-stage Runge--Kutta (RK) method} with $s \in \N$ is given by \cite{hairer_solving_1993}
\begin{subequations}\label{eq:RKF}
\begin{align}
x_{n+1} &= x_n + h_n \sum_{i=1}^s b_i k_{n,i}, \label{eq:RKF-a}\\
k_{n,i} &= f(X_{n,i}, p, t_n +c_ih_n),\label{eq:RKF-b}\\
X_{n,i} &= x_n + h_n \sum_{j=1}^s a_{ij}k_{n,j}\label{eq:RKF-c}, 
\end{align}
\end{subequations}
where $h_n = t_{n+1} - t_n$ in \eqref{eq:RKF-a} denotes a step-size.
The coefficients $a_{ij}, b_i, c_i \in \mathbb R$ can be arranged in a so-called Butcher tableau (Fig.~\ref{fig:butcher}), 
with entries $a_{ij}$ defining the Runge--Kutta matrix $A$.

\begin{figure}[h!]
\begin{center}
$\begin{array}{c|ccccc}	
c_1		& a_{11} 	& a_{12}	& \dots 	& a_{1s}\\
c_2		& a_{21}	& a_{22}	& \dots 	& a_{2s}\\
\vdots 	& \vdots 	& \vdots 	&\ddots 	& \vdots\\
c_s 		& a_{s1} 	& a_{s2}	& \dots 	& a_{ss}\\
\hline
& b_1	& b_2 	& \dots 	& b_{s}
\end{array}=
\begin{array}{c|c}
{c}& A\\
\hline\\[-3mm]
       & {b^T} \\
\end{array}$\\[3mm]

\begin{tabular}{c|ccccc}	
$c_1$ &&&&&\\
$c_2$ & $a_{21}$ &&&&\\
$c_3$ & $a_{31}$ & $a_{32}$&&&\\
$\vdots$ &\vdots&\vdots &$\ddots$&&\\
$c_s$ & $a_{s1}$ & $a_{s2}$&\dots&$a_{ss-1}$&\\
\hline
& $b_1$& $b_2$& \dots & $b_{s-1}$ & $b_{s}$
\end{tabular}
\end{center}
\caption{\textsc{Above}: Butcher tableau of a general $s$-stage Runge--Kutta method. \textsc{Below}: Butcher tableau of a $s$-stage explicit Runge--Kutta method.}
\label{fig:butcher}
\end{figure}

Lower-triangular Runge--Kutta matrices $A$, i.e.
\begin{equation}\label{eq:explicit-RK}
a_{ij} = 0\quad \text{for}\quad j \ge i,
\end{equation}
result in \textit{explicit} RK schemes, and in \textit{implicit} RK schemes otherwise. Implicit Runge--Kutta methods are well-suited for integrating numerically stiff ODEs, but are also 
significantly more complex than explicit ones. Since \eqref{eq:RKF-b} can not be solved explicitly, a system of algebraic equations has to be solved. The following theorem specifies
 the conditions under which a solution for these equations exists.

\begin{theorem}[\bfseries Existence of a Numerical Solution; {\cite[Ch. II, Thm. 7.2]{hairer_solving_1993}}]
For any $p \in \R^{n_p}$ let $f(\cdot,p,\cdot)$ of \eqref{ivp-b} be continuous and satisfy a Lipschitz condition on $\mathbb R^{n_x} \times [0, T]$ with constant $L$, independent of $p$. If 
\begin{equation}\label{eq:step-size-condition}
h < \frac{1}{L \max_{i=1, \dots, s}\sum_{j=1}^s |a_{ij}|}
\end{equation}
there exists a unique solution of \eqref{eq:RKF}, which can be obtained by {fixed-point iteration}. If $f(x,p,t)$ is $q$ times differentiable, the functions $k_{n,i}$ (as functions of $h$) are also in $C^q$.
\begin{proof}
A detailed proof can be found in \cite[Ch. II, Thm. 7.2]{hairer_solving_1993}.
\end{proof}

\end{theorem}

Suppose that the given system \eqref{ivp-b} is \textit{partitioned} into two parts with $x = (\textcolor{black}{y}^\T, \textcolor{black}{z}^\T)^\T$, $f = (f_1^\T, f_2^\T)^\T$ and
\begin{subequations}\label{eq:partitioned-system}
\begin{align}
\dot {\textcolor{black}{y}} &= f_1({\textcolor{black}{y}}, {\textcolor{black}{z}}, t),\label{eq:partitioned-system-a}\\
\dot {\textcolor{black}{z}} &= f_2({\textcolor{black}{y}}, {\textcolor{black}{z}}, t).\label{eq:partitioned-system-b}
\end{align}
\end{subequations}
\textit{Partitioned Runge--Kutta (PRK) methods} integrate \eqref{eq:partitioned-system} using 
two different sets of coefficients
\begin{subequations}\label{eq:partitioned-coeffs}
\begin{align}
a_{ij}, b_i, c_i \in \mathbb R \; \text{ for } \; \eqref{eq:partitioned-system-a},\\
 \ol{a}_{ij}, \ol b_i, \ol c_i \in \mathbb R \; \text{ for } \; \eqref{eq:partitioned-system-b}.
\end{align}
\end{subequations}
The following theorems state conditions under which RK methods preserve certain quantities that should be invariant under the flow of the system that is  integrated numerically. In this sense, such RK schemes are called \textit{symplectic}.

\begin{theorem}[{\bfseries Symplectic Runge--Kutta Method; \cite[Ch. VI, Thm. 7.6 and 7.10]{Hairer:2006}}]
Assume that the system, \eqref{ivp-b} has a quadratic invariant $I$, i.e. $I(\cdot, \cdot)$ is a real-valued bilinear mapping such that $(d / d t) I(x(t), x(t)) = 0$, for 
each $t$ and $x_0$. If the coefficients of a Runge--Kutta method \eqref{eq:RKF} satisfy
\begin{equation}\label{eq:symplectic-condition}
b_i a_{ij} + b_j a_{ji} - b_ib_j = 0,
\end{equation}
then the value $I(x_n, x_n)$ does not depend on $n$.
\end{theorem}

\begin{theorem}[{\bfseries Symplectic PRK Method; \cite[Thm. 2.4 and 2.6]{Sanz-Serna:2016aa}}]
Assume that $S(\cdot, \cdot)$ is a real-valued bilinear mapping such that $(d / d t) S(y(t), z(t)) = 0$ for each $t$ and $x_0$ of the solution $x(t) = [y(t)^{\top}, z(t)^{\top}]^{\top}$ of \eqref{eq:partitioned-system}. If the coefficients of the partitioned Runge--Kutta method \eqref{eq:partitioned-coeffs} satisfy
\begin{equation}\label{eq:def-symplectic-PRK}
b_i\ol a_{ij} - b_i\ol b_j + \ol b_ja_{ji} = 0, \quad \ol b_i = b_i, \quad \ol c_i = c_i,
\end{equation}
then the value $S(y_n, z_n)$ does not depend on $n$.
\end{theorem}

\begin{remark}\label{rem:construct-sPRK}
Assume the first set of Runge--Kutta coefficients are given and denoted by $a_{ij}, b_i, c_i$ with indices $i,j \in [s]$. This method is used for the first $n$-variables \eqref{eq:partitioned-system-a}. Furthermore, let $b_i \ne 0$ for all stages $i \in [s]$. In view of 
condition \eqref{eq:def-symplectic-PRK}, we can construct a \textit{symplectic} PRK method by choosing
\begin{equation}\label{eq:recipe}
\ol a_{ij} := b_j - b_ja_{ji}/b_i, \quad \ol b_i := b_i, \quad \ol c_i := c_i,
\end{equation}
as coefficients for the second $n$-variables \eqref{eq:partitioned-system-b}. This construction results in an overall \textit{symplectic PRK method} of the partitioned system \eqref{eq:partitioned-system}.
\end{remark}
 

\subsection{Computing Adjoint Sensitivities} \label{sec:Solving-Para-Prob}
There are two basic approaches for computing \eqref{eq:adjoint-sensitivity-column-form}, the \textit{differentiate-then-discretize} approach and the
 \textit{discretize-then-differentiate} approach. Figure \ref{fig:sensitivity-diagram} illustrates both approaches by paths colored with black and violet, respectively. Details are worked out in this section. Our main objective is to make this diagram \textit{commutative} by adopting a class of numerical schemes as outlined in the preceding section.

\begin{figure*}[h]
\begin{center}
\begin{tikzcd}[column sep=large]
\normalsize
\textbf{dynamical system}
\arrow[r,"\textcolor{blue}{\text{differentiate}}"]
\arrow[d,"\textcolor{violet}{\text{discretize}}",swap]
  &\normalsize \textcolor{blue}{\textbf{adjoint system}} 
    \arrow[d,"\textcolor{blue}{\text{discretize}}"]
\\
\normalsize\textcolor{violet}{\textbf{discretization}}
\arrow[r,"\textcolor{violet}{\text{differentiate}}",swap]
  &\normalsize \textbf{numerical sensitivity}
\end{tikzcd}
\end{center}
\caption{Illustration of the methodological part of this section. Our approach satisfies the commuting diagram, i.e.~\textit{identical} results are obtained either if the continuous problem is differentiated first and than discretized (black path), or the other way around (violet path).
}
\label{fig:sensitivity-diagram}
\end{figure*}
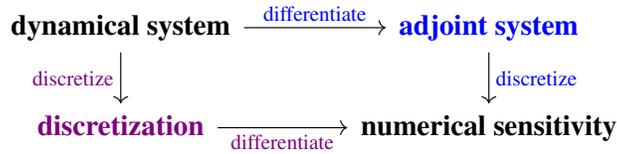

In the following, we drop the dependency of $x(t)$ on the parameter $p$, to simplify notation by just writing $x(t)$.
The following theorem details the black path of Figure \ref{fig:sensitivity-diagram}. 
\begin{theorem}[\bfseries Adjoint Sensitivity: Differentiate-then-Discretize]\label{thm:sensitivity-continuous}
The gradient \eqref{eq:adjoint-sensitivity-column-form} of the objective function \eqref{eq:def-Phi} $\Phi(p) = C(x(T))$ of problem \eqref{eq:parameter-problem-general} with respect to the parameter $p$ is given by
\begin{equation}\label{eq:continuous-sensitivity-parameter}
\Egrad \Phi(p) = \int_0^T d_p f(x(t),p,t)^{\T} \lambda(t) dt,
\end{equation}
where $x(t), \lambda(t)$ solve the two-point boundary value problem
\begin{subequations}\label{eq:two-point-boundary}
\begin{align}
\dot x(t) &= f(x(t),p,t),  &x(0) &= x_0,\label{eq:two-point-boundary-a}\\
\dot \lambda(t) &= -d_x f(x(t),p,t)^{\T} \lambda(t), &\lambda(T) &= \Egrad \mc{C}(x(T)).\label{eq:two-point-boundary-b}
\end{align}
\end{subequations}
In terms of the pseudo-Hamiltonian 
\begin{equation}\label{eq:def-pseudo-Hamiltonian}
H(x,\lambda,p,t) = \la f(x, p, t), \lambda \ra, 
\end{equation}
the system has the following form
\begin{subequations}\label{eq:continuous-first-order-conditions}
\begin{align}
\dot x(t) &= \phantom{-}d_\lambda H(x,\lambda,p,t) ,  &x(0) &= x_0,\\
\dot \lambda(t) &= -d_x H(x,\lambda,p,t) , &\lambda(T) &= \Egrad \mc{C}(x(T)).
\end{align}
\end{subequations}
\end{theorem}
\begin{proof}
See Appendix~\ref{appendix-proof-sensitivity-continuous}.
\end{proof}

\begin{remark}\label{rem:diff-discr}
The presence of the pseudo-Hamiltonian \eqref{eq:def-pseudo-Hamiltonian} suggests to use either a \textit{symplectic} RK method or a \textit{symplectic} PRK method to integrate 
the boundary value problem \eqref{eq:two-point-boundary}. In view of Remark~\ref{rem:construct-sPRK}, we can use a general RK method with coefficients $a_{ij}, b_i, c_i$ for $i,j \in [s]$ for the first variables \eqref{eq:two-point-boundary-a}, and another RK method with $\ol a_{ij},\ol b_i,\ol c_i$ for $i,j \in[s]$ satisfying \eqref{eq:recipe} for the second variables
\eqref{eq:two-point-boundary-b}. Again, this construction results in an overall \textit{symplectic PRK method} of the boundary problem \eqref{eq:two-point-boundary}. \textcolor{black}{Note that \eqref{eq:two-point-boundary-a} is independent of variable $\lambda$. Due to this property we can solve \eqref{eq:two-point-boundary} sequentially in practice, i.e. we first integrate \eqref{eq:two-point-boundary-a} and afterwards \eqref{eq:two-point-boundary-b}.}
\end{remark}

Now we consider the alternative violet path of Figure \ref{fig:sensitivity-diagram}. 
Applying a RK method with step-sizes $h_n = t_{n+1} - t_n>0$ to problem \eqref{eq:parameter-problem-general} results in the \textit{nonlinear optimization problem} 
\begin{subequations}\label{eq:resulting-NLP-general}
\begin{align}
\min_{p \in \mc{P}}\; & \mc{C}\big( x_N(p) \big )\\
\text{s.t.}\quad & x_{n+1} = x_n + h_n \sum_{i=1}^s b_i k_{n,i},\label{eq:NLP-general-a}\\
\quad & k_{n,i} = f( X_{n,i}, p,  t_n +c_ih_n), \quad i \in [s], \label{eq:NLP-general-b}\\
\quad & X_{n,i} = x_n + h_n \sum_{j=1}^s a_{ij}k_{n,j}, \quad i \in [s], \label{eq:NLP-general-c}\\
\quad & x_0 = x(0),\label{eq:NLP-general-d}
\end{align}
\end{subequations}
with $n = 0, \dots, N-1$.

Next, we \textit{differentiate} this problem and state the result in the following theorem.
\begin{theorem}[\bfseries Adjoint Sensitivity: Discretize-then-\\Differentiate Approach]\label{theorem:discrete-sensitivity}
Suppose the step-size $h_n$ satisfies condition \eqref{eq:step-size-condition}.
Then, the gradient of the objective function $\Phi(p) = \mc{C}(x_N(p))$ from \eqref{eq:resulting-NLP-general} with respect to parameter $p$ is given by
\begin{equation}\label{eq:discrete-sensitivity-parameter}
\Egrad \Phi(p) = \sum_{n=0}^{N-1} h_{n} \sum_{i=1}^{s} \ol b_i \left (d_p f(X_{n,i}, p, t_n +\ol c_ih_n) \right )^\T \Lambda_{n,i},
\end{equation}
where the discrete adjoint variables are given by
\begin{subequations}\label{eq:discrete-adjoints-RK}
\begin{align}
\lambda_{n+1} &= \lambda_{n} + h_n \sum_{i=1}^s \ol b_i \ell_{n,i},\label{eq:discrete-adjoints-RK-a}\\
\ell_{n,i} &= - d_x f(X_{n,i}, p, t_n + \ol c_i h_n)^\T \Lambda_{n,i}, \label{eq:discrete-adjoints-RK-b}\\
\Lambda_{n,i} &= \lambda_{n} + h_n \sum_{j=1}^s \ol a_{ij} \ell_{n,j},\label{eq:discrete-adjoints-RK-c}
\end{align}
\end{subequations}
with $n = 0, \dots, N-1$, $i \in [s]$ and step-size $h_n = t_{n+1} - t_n$. The internal stages $X_{n,i}$ are given by \eqref{eq:NLP-general-c}.
This scheme is a general RK method \eqref{eq:RKF} applied to the adjoint system \eqref{eq:two-point-boundary-b} with coefficients
\begin{equation}\label{eq:discr-sens-coefficient}
\ol a_{ij} = b_j - \frac{a_{ji} b_j}{b_i}, \quad \ol b_i = b_i, \quad \ol c_i = c_i,
\end{equation}
for $b_i \ne 0$ and $i,j=[s]$. 
\begin{proof}
An outline of the proof can be found in \cite[Thm. 3.6]{Sanz-Serna:2016aa}. Following the suggested outline, we provide a detailed proof in Appendix~\ref{appendix:theorem-discrete-sensitivity}.
\end{proof}
\end{theorem}

\begin{remark}\label{rem:commuting}
Comparing the statements of Theorem~\ref{thm:sensitivity-continuous} and Theorem~\ref{theorem:discrete-sensitivity}, we see that the formula of the discrete sensitivity 
\eqref{eq:discrete-sensitivity-parameter} is an approximation of the integral \eqref{eq:continuous-sensitivity-parameter} with quadrature weights $b_i$. 
Furthermore, we observe that the coefficients of the constructed PRK method \eqref{eq:recipe} coincide with the derived coefficients \eqref{eq:discr-sens-coefficient}. 
Thus, by restricting the class of numerical schemes to \textit{symplectic PRK methods} satisfying \eqref{eq:def-symplectic-PRK}, the approaches due to the Theorem \ref{thm:sensitivity-continuous} (and Remark \ref{rem:diff-discr}) and Theorem \ref{theorem:discrete-sensitivity} are 
mathematically identical, and the diagram depicted by Figure \ref{fig:sensitivity-diagram} \textit{commutes}.
\end{remark}

\subsection{Two Specific Numerical Schemes}\label{sec:two-schemes}
{\color{black} We complement and illustrate the general result of the preceding section by specifying two numerical RK schemes of different order, the basic explicit Euler method and Heun's method, respectively.}

\subsubsection{Adjoint Sensitivity: Explicit Euler method}\label{sec:adjoint-Euler}
We integrate the forward dynamic \eqref{eq:two-point-boundary-a} with the explicit Euler method \cite{Hairer:2008aa}. The straightforward use of \eqref{eq:recipe} leads to
another Runge--Kutta method for integrating the adjoint system \eqref{eq:two-point-boundary-b}. The forward and backward coefficients of this overall symplectic partitioned Runge--Kutta method
are then given by Table~\ref{tab:euler}.
\begin{table}[h!]
\caption{Symplectic PRK coefficients induced by the explicit Euler method.}\label{tab:euler}

\begin{center}
\begin{tabular}{c}
\begin{tabular}{c|c}
$c_1$ & $a_{11}$ \\ \hline
& $b_1$
\end{tabular}  =
\begin{tabular}{c|c}
$0$ & \\ \hline
& $1$
\end{tabular}
\\[3mm]
forward coefficients\\[5mm]
\begin{tabular}{c|c}
$\ol c_1$ & $\ol a_{11}$ \\ \hline
& $\ol b_1$
\end{tabular}  =
\begin{tabular}{c|c}
$0$ & $1$ \\ \hline
& $1$
\end{tabular}
\\[3mm]
backward coefficients
\end{tabular}

\end{center}

\end{table}

By substituting the backward coefficients $\ol a_{11}, \ol b_1$ and $\ol c_1$ into \eqref{eq:discrete-adjoints-RK}, we derive the concrete formulas of the discrete adjoint method
\begin{subequations}\label{eq:discrete-backward-expl-euler}
\begin{align}
\lambda_{n+1} &= \lambda_{n} + h_n \ell_{n,1} \label{eq:discrete-backward-expl-euler-a}\\
\ell_{n,1} &= - \partial_x f(X_{n,1}, t_n)^\T \Lambda_{n,1}\\
\Lambda_{n,1} &= \lambda_{n} + h_n \ell_{n,1}.\label{eq:discrete-backward-expl-euler-c}
\end{align}
\end{subequations}
Note, that \eqref{eq:discrete-backward-expl-euler-c} coincides with \eqref{eq:discrete-backward-expl-euler-a}, that is by traversing from $n+1$ to $n$, we can rewrite \eqref{eq:discrete-backward-expl-euler} in the form  
\begin{equation}\label{eq:identities}
\lambda_{n} = \lambda_{n+1} + h_n d_x f(X_{n,1}, t_n)^\T  \lambda_{n+1}.
\end{equation}
Formula \eqref{eq:discrete-sensitivity-parameter} for the gradient of $\Phi(p) = \mc{C}(x_N(p))$ from \eqref{eq:resulting-NLP-general} reads
\begin{equation}\label{eq:euler-discrete-sensitivity-parameter}
\Egrad \Phi(p) = \sum_{n=0}^{N-1} h_{n} d_p f(X_{n,1},t_n)^\T  \lambda_{n+1}.
\end{equation}

\subsubsection{Adjoint Sensitivity: Heun's method}\label{sec:adjoint-Heun}
We integrate the forward dynamic \eqref{eq:two-point-boundary-a} with Heun's method \cite{Hairer:2008aa}. The straightforward use of \eqref{eq:recipe} leads to
another Runge--Kutta method for integrating the adjoint system \eqref{eq:two-point-boundary-b}. The forward and backward coefficients of this overall symplectic partitioned Runge--Kutta method
are then given by Table~\ref{tab:heun}
 \begin{table}[h!]
 \caption{Symplectic PRK coefficients induced by Heun's method.}\label{tab:heun}
\begin{center}
\begin{tabular}{c}
\begin{tabular}{c|cc}
$c_1$ & $a_{11}$ & $a_{12}$\\ 
$c_2$ & $a_{21}$ & $a_{22}$ \\ \hline
& $b_1$ & $b_2$
\end{tabular}=
\begin{tabular}{c|cc}
$0$ & \\ 
$1$ & $1$ \\ \hline
& $1/2$ & $1/2$
\end{tabular}
\\[3mm]
forward coefficients\\[5mm]
\begin{tabular}{c|cc}
$\ol c_1$ & $\ol a_{11}$ & $\ol a_{12}$\\ 
$\ol c_2$ & $\ol a_{21}$ & $\ol a_{22}$ \\ \hline
& $\ol b_1$ & $\ol b_2$
\end{tabular}=
\begin{tabular}{c|cc}
$0$ & $1/2$ & -$1/2$\\ 
$1$ & $1/2$ & $1/2$\\ \hline
& $1/2$ & $1/2$
\end{tabular}
\\[3mm]
backward coefficients
\end{tabular}

\end{center}
\label{table:heun}
\end{table}

Although the butcher tableau of the backward coefficients (see Table~\ref{table:heun}, right matrix) is completely dense, the final 
update formulas are \textit{explicit} \textcolor{black}{by traversing backwards in time}, as we will show below.
Again, we derive the concrete formulas of the discrete adjoint method by substituting the backward coefficients into \eqref{eq:discrete-adjoints-RK}
\begin{subequations}\label{eq:discrete-backward-heun}
\begin{align}
\lambda_{n+1} &= \lambda_{n} + h_n \big ( \tfrac{1}{2} \ell_{n,1} + \tfrac{1}{2} \ell_{n,2} \big )\label{eq:discrete-backward-heun-a}\\
\ell_{n,1} &= - d_x f(X_{n,1}, t_n)^\T \Lambda_{n,1}\\
\ell_{n,2} &= - d_x f(X_{n,2}, t_n + h_n)^\T \Lambda_{n,2}\\
\Lambda_{n,1} &= \lambda_{n} + h_n \big ( \tfrac{1}{2} \ell_{n,1} - \tfrac{1}{2} \ell_{n,2} \big )\label{eq:discrete-backward-heun-d}\\
\Lambda_{n,2} &= \lambda_{n} + h_n \big ( \tfrac{1}{2} \ell_{n,1} + \tfrac{1}{2} \ell_{n,2} \big ). \label{eq:discrete-backward-heun-e}
\end{align}
\end{subequations}
Note, that \eqref{eq:discrete-backward-heun-e} coincides with \eqref{eq:discrete-backward-heun-a}, which implies the equations
\begin{subequations}\label{eq:identities}
\begin{align}
\lambda_{n+1} &= \Lambda_{n,2}, \quad \text{and}\\
\ell_{n,2} &= - d_x f(X_{n,2}, t_n + h_n)^\T \lambda_{n+1}.
\end{align}
\end{subequations}
Using \eqref{eq:identities}, we reformulate \eqref{eq:discrete-backward-heun-d}
\begin{subequations}
\begin{align}
&\Lambda_{n,1} = \lambda_{n} + h_n \big ( \tfrac{1}{2} \ell_{n,1} - \tfrac{1}{2} \ell_{n,2} \big ) = \lambda_{n} + h_n \big ( \tfrac{1}{2} \ell_{n,1} - \tfrac{1}{2} \ell_{n,2} \big )+ (h_n \ell_{n,2}- h_n \ell_{n,2})\\
&= \lambda_{n} + h_n \big ( \tfrac{1}{2} \ell_{n,1} + \tfrac{1}{2} \ell_{n,2} \big ) - h_n \ell_{n,2} \stackrel{\eqref{eq:discrete-backward-heun-a}}{=} \lambda_{n+1} - h_n \ell_{n,2}\\
&\stackrel{\eqref{eq:identities}}{=} \lambda_{n+1} + h_n d_x f(X_{n,2}, t_n + h_n)^\T \lambda_{n+1}.\label{eq:expl-backwar-step}
\end{align}
\end{subequations}
Formula \eqref{eq:expl-backwar-step} is an explicit Euler step traversing backwards from $n+1$ to $n$. Thus, we can rewrite the overall scheme \eqref{eq:discrete-backward-heun} as
\begin{subequations}
\begin{align}
\tilde \lambda_n &= \lambda_{n+1} + h_n d_x f(X_{n,2}, t_n + h_n)^\T \lambda_{n+1}\\
\begin{split}
\lambda_n &= \lambda_{n+1} + \frac{h_n}{2} \Big (d_x f(X_{n,1}, t_n)^\T \tilde \lambda_{n} + d_x f(X_{n,2}, t_n + h_n)^\T \lambda_{n+1} \Big ).
\end{split}
\end{align}
\end{subequations}
Again, this is an \textit{explicit method traversing backwards from $n+1$ to $n$}. Formula \eqref{eq:discrete-sensitivity-parameter} for the gradient of $\Phi(p) = \mc{C}(x_N(p))$ from \eqref{eq:resulting-NLP-general} has the form 
\begin{equation}\label{eq:heun-discrete-sensitivity-parameter}
\begin{split}
\partial_p \mc{C}(x_N) &= \sum_{n=0}^{N-1} \frac{h_{n}}{2} \Big (d_p f(X_{n,1},t_n)^\T \tilde \lambda_n  + d_p f(X_{n,2},t_n + h_n)^\T \lambda_{n+1} \Big ).
\end{split}
\end{equation}

\begin{remark}\label{rem:explicit-backwards}
\textcolor{black}{Both example schemes (explicit Euler \& Heun's method) have in common that the final update schemes of the 
adjoint integration can be solved \textit{explicitly} by traversing backwards from $n+1 \to n$. Note that this holds for these two specific numerical schemes, but may not hold in general for other higher-order schemes.}
\end{remark}

\section{Image Labeling using Geometric Assignment}
\label{sec:Assignment-Flow}
In this section, we summarize material from \cite{Astroem2017} and \cite{Zeilmann:2018aa} required in the remainder of this paper. See also \cite{Schnorr:2019aa} for a discussion of the assignment flow approach in a broader context.

Let $\mc{G}=(\mc{V},\mc{E})$ be a given undirected graph with $m := |\mc{V}|$ vertices and let
\begin{equation}\label{eq:def-FI}
\begin{split}
 &f \colon \mc{V} \to \mc{F},\quad i \mapsto f_i \in \mc{F} \quad \text{with}\\ &f(\mc{V}) =: \mc{F}_{\mc{V}} \subset \mc{F}
\end{split}
\end{equation}
be data on the graph given in a metric space $(\mc{F},d)$. We call $\mc{F}_{\mc{V}}$ \textit{image data} given by \textit{features} $f_{i}$ extracted from a raw image at pixel $i \in \mc{V}$ in a preprocessing step. Along with $f$ we assume prototypical data
\begin{equation}\label{eq:def-GJ}
\mc{X} = \big\{ \ell_{1}, \dots, \ell_n \big\} \subset \mc{F}
\end{equation}
to be given, henceforth called \textit{labels}. Each label $\ell_{j}$ represents the data of class $j$. \textit{Image labeling} denotes the problem of finding an assignment $\mc{V} \to \mc{X}$  assigning class labels to nodes depending on the image data $\mc{F}_{\mc{V}}$ and the local context encoded by the graph structure $\mc{G}$. We refer to \cite{Huhnerbein:2018aa} for more details and background on the image labeling problem.

$\mc{G}$ may be a grid graph (with self-loops) as in low-level image processing or a less structured graph, with arbitrary connectivity in terms of the neighborhoods
\begin{equation}\label{eq:def-neighborhood-i}
\mc{N}_{i} = \{k \in \mc{V} \colon ik=ki \in \mc{E}\} \cup \{i\},\quad i \in \mc{V},
\end{equation}
where $ik$ is a shorthand for the undirected edge $\{i, k\} \in \mc{E}$. We require these neighborhoods to satisfy the relations
\begin{equation}\label{eq:Ni-Nk}
k \in \mc{N}_{i} \quad\gdw\quad i \in \mc{N}_{k},\qquad\forall i,k \in \mc{I}.
\end{equation}
We associate with each neighborhood $\mc{N}_{i}$ from \eqref{eq:def-neighborhood-i} weights $\omega_{ik} \in \R$ for all $k \in \mc{N}_i$, satisfying
\begin{equation}\label{eq:def-weights}
\omega_{ik} > 0 \quad \text{and} \quad \sum_{k \in \mc{N}_{i}} \omega_{ik} = 1, \quad\text{for all} \; i \in \mc{V}.
\end{equation}
These weights parametrize the regularization property of the assignment flow below. Learning these weights from given data is the subject of the remainder of this paper.

\subsection{Assignment Manifold}
\label{sec:AM}

The probabilistic assignment of labels $\mc{X}$ at one node $i \in \mc{V}$ are represented by the manifold of discrete probability distributions with full support
\begin{equation}\label{eq:def-S}
  \mc{S}_n := \{p \in \Delta_{n} \colon p > 0\}
\end{equation}
with constant tangent space for all $p \in \mc{S}_n$
\begin{equation}
  T_p \mc{S}_n = \{v \in \R^{n} \colon \la \eins, v \ra = 0\} =: T_n.
\end{equation}
Throughout this paper, we only work with $T_n$. The probability space $\mc{S}$ is turned into a Riemannian manifold $(\mc{S}_n, g)$ by equipping it with the Fisher-Rao (information) metric
\begin{equation}\label{eq:def-gp-T0}
g_{p}(u,v) := \sum_{j \in [n]} \frac{u_{j} v_{j}}{p_{j}},
\end{equation}
with $u, v \in T_n$ and $p \in \mc{S}_n$.
Furthermore, we have the uniform distribution of labels
\begin{equation}\label{eq:barycenter-S}
\BS{n} := \frac{1}{n}\eins_{n} \in \mc{S}_n,
\qquad(\textit{barycenter})
\end{equation}
and the orthogonal projection onto the tangent space with respect to the standard Euclidean structure of $\R^n$
\begin{equation}\label{eq:def-Pi0}
\PT_n \colon \R^{n} \to T_n,\quad
\PT_n := I - \eins_{\mc{S}_n}\eins^{\T}
\end{equation}
with $\ker(\PT_n) = \R \eins_n$. The \textit{replicator operator} for $p \in \mc{S}_n$ is given by the linear map
\begin{equation}\label{eq:def-Rp}
\RO_{p} \colon \R^{n} \to T_n,\quad
\RO_{p} := \Diag(p)-p p^{\T},
\end{equation}
satisfying
\begin{equation}\label{eq:Rp-Pi0}
\RO_{p} = \RO_{p} \PT_n = \PT_n \RO_{p}.
\end{equation}
The \textit{Riemannian gradient} of a smooth function $f \colon \mc{S}_n \to \R$ is denoted by $\Rgrad f \colon \mc{S}_n \to T_n$ and relates to the Euclidean gradient $\Egrad f$ by \cite[Prop.~1]{Astroem2017} as
\begin{equation}\label{eq:relation_Rgrad_and_Egrad}
  \Rgrad f(p) = \RO_p \Egrad f(p)\qquad \text{for } p \in \mc{S}_n.
\end{equation}

Adopting the $\alpha$-connection with $\alpha=1$, also called \textit{e-connection}, from information geometry \cite[Section 2.3]{Amari:2000aa}, \cite{Ay:2017aa}, the exponential map based on the corresponding affine geodesics reads
\begin{subequations}
\begin{align}
\begin{split}\label{eq:Exp-e}
  \Exp &\colon \mc{S}_n \times T_n \to \mc{S}_n,\\
   (p,v) &\mapsto \Exp_{p}(v) = \frac{pe^{\frac{v}{p}}}{\la p, e^{\frac{v}{p}}\ra}
  \end{split}\\
  \begin{split}\label{eq:def-invExp-e}
  \Exp^{-1} &\colon \mc{S}_n \times \mc{S}_n \to T_n, \\
   (p,q) &\mapsto \Exp_{p}^{-1}(q) = \RO_{p}\log\frac{q}{p}.
    \end{split}
\end{align}
\end{subequations}
Specifically, we define for all $p \in \mc{S}_n$
\begin{subequations}
\begin{align}
\begin{split}\label{eq:def-exp}
  \exp_{p} &\colon T_n \to \mc{S}_n,\\
   z &\mapsto \Exp_{p} \circ \RO_{p}(z) = \frac{p e^{z}}{\la p, e^{z}\ra},
  \end{split}\\
  \exp_{p}^{-1} &\colon \mc{S}_n \to T_n,\quad q \mapsto \PT_n\log\frac{q}{p}.\label{eq:def-invexp}
\end{align}
\end{subequations}
Applying the map $\exp_{p}$ to a vector in $\R^{n} = T_n \oplus \R\eins_n$ does not depend on the constant component of the argument, due to \eqref{eq:Rp-Pi0}.

\begin{remark}\label{rem:Exp-map}
The map $\Exp$ corresponds to the e-connection of information geometry \cite{Amari:2000aa}, rather than to the exponential map of the Riemannian connection. Accordingly, the affine geodesics \eqref{eq:Exp-e} are not length-minimizing with respect to the Riemannian structure. But locally, they provide a close approximation \cite[Prop.~3]{Astroem2017} and are more convenient for numerical computations.
\end{remark}

Global label assignments on the whole set of nodes $\mc{V}$ are represented as points on the \textit{assignment manifold}, given by the product
\begin{equation}
\mc{W} := \mc{S}_n \times \dotsb \times \mc{S}_n \qquad(m = |\mc{V}|\; \text{times})
\end{equation}
with constant \textit{tangent space}
\begin{equation}
\mc{T}_{\mc{W}} := T_n \times \dotsb \times T_n \qquad(m = |\mc{V}|\; \text{times})
\end{equation}
and Riemannian structure $(\mc{W}, g)$ given by the Riemannian product metric. We identify $\mc{W}$ with the embedding into $\R^{m\times n}$
\begin{equation}\label{eq:mcW-matrix-embed}
\begin{split}
  &\mc{W} = \{ W \in \R^{m\times n} \colon W\eins_n = \eins_m \text{ and } W_{ij} > 0\\
  &\qquad \quad\text{for all } i\in[m], \; j\in [n]\}.
  \end{split}
\end{equation}
Thus, points $W \in \mc{W}$ are row-stochastic matrices $W \in \R^{m \times n}$ with row vectors $W_{i} \in \mc{S}_n,\; i \in \mc{V}$ representing the label assignments for every $i \in \mc{V}$. Due to this embedding of $\mc{W}$, the tangent space $\mc{T}_{\mc{W}}$ can be identified with
\begin{equation}\label{eq:mcT-matrix-embed}
  \mc{T}_{\mc{W}} = \{ V \in \R^{m\times n} \colon V\eins_n = 0\}
\end{equation}
and therefore for $V \in \mc{T}_{\mc{W}}$ every row vector $V_i$ is contained in $T_n$ for every $i\in \mc{V}$. The global uniform distribution, given by the uniform distribution in every row, again called \textit{barycenter}, is denoted by
\begin{equation}
\BW := (\eins_{\mc{S}_n},\dotsc,\eins_{\mc{S}_n}) = \eins_m \BS{n}^\T \in \mc{W},
\end{equation}
where the second equality is due to the embedding \eqref{eq:mcW-matrix-embed}. The mappings \eqref{eq:def-Pi0}-\eqref{eq:Exp-e} naturally extend to the assignment manifold $\mathcal{W}$
\begin{subequations}\label{eq:assignment-flow-maps}
\begin{align}
  \PT[Z] &= \big(\PT_n[Z_{1}],\dotsc,\PT_n[Z_{m}]\big)^\T \in \mc{T}_{\mc{W}}, \label{eq:Pi0-W} \\
  \RO_{W}[Z] &= \big(\RO_{W_{1}}[Z_{1}],\dotsc, \RO_{W_{m}}[Z_{m}]\big)^\T \in \mc{T}_{\mc{W}}, \label{eq:R-W} \\
  \Exp_{W}(V)
  &= \big(\Exp_{W_{1}}(V_{1}),\dotsc,\Exp_{W_{m}}(V_{m})\big)^\T \in \mc{W},\label{eq:Exp-e-W}
\end{align}
\end{subequations}
with $W\in \mc{W}$, $Z \in \R^{m \times n}$ and $V \in \mc{T}_{\mc{W}}$. The maps $\exp_{W}, \Exp^{-1}_{W}, \exp^{-1}_{W}$ are similarly defined based on \eqref{eq:def-exp}, \eqref{eq:def-invExp-e} and \eqref{eq:def-invexp}. Due to \eqref{eq:relation_Rgrad_and_Egrad} , the Riemannian gradient and the Euclidean gradient of a smooth function $f \colon \mc{W} \to \R$ are also related by
\begin{equation}\label{eq:relation_Rgrad_and_Egrad-W}
  \Rgrad f(W) = \RO_W[\partial f(W)]\quad \text{for } W \in \mc{W}.
\end{equation}

\subsection{Assignment Flow}
\label{sec:AF}

Based on the given data \eqref{eq:def-FI} and labels \eqref{eq:def-GJ}, the $i$-th row of the \textit{distance matrix} $D \in \R^{m \times n}$ is defined by
\begin{equation}\label{eq:def-distance-matrix}
  D_{i} :=\big(d(f_{i},\ell_{1}),\dotsc,d(f_{i}, \ell_{n})\big)^{\T}\in \R^n,
\end{equation}
for all $i \in \mc{V}$. This distance information is \textit{lifted} onto the manifold by the following \textit{likelihood matrix}
\begin{subequations}\label{eq:def-LW}
\begin{align}
L(W) :=& \exp_{W}(-D/\rho) \in \mc{W},
\\ \label{eq:def-LiWi}
L_{i}(W_{i})
=& \frac{W_{i} e^{-\frac{1}{\rho} D_{i}}}{\la W_{i},e^{-\frac{1}{\rho} D_{i}} \ra},\quad \rho > 0,\quad i \in \mc{V},
\end{align}
\end{subequations}
where $\rho>0$ is a scaling parameter to normalize the a-priori unknown scale of the distances induced by the features $f_{i}$ depending on the application at hand. This representation of the data is regularized by weighted geometric averaging in the local neighborhoods \eqref{eq:def-neighborhood-i} using the weights \eqref{eq:def-weights}, to obtain the \textit{similarity matrix} $S(W) \in \mc{W}$, with $i$-th row defined by
\begin{equation}\label{eq:def-SW}
\begin{split}
&S_i \colon \mc{W} \to \mc{S}_n,\\
&S_{i}(W) := \Exp_{W_{i}} \Big(\sum_{k \in \mc{N}_{i}} w_{ik} \Exp_{W_{i}}^{-1}(L_{k}(W_k))\Big).
\end{split}
\end{equation}
If $\Exp_{W_{i}}$ were the exponential map of the Riemannian (Levi-Civita) connection, then the sum inside the outer brackets of the right-hand side in \eqref{eq:def-SW} would just be the negative Riemannian gradient with respect to $W_{i}$ of the objective function used to define the Riemannian center of mass, i.e.~the weighted sum of the squared Riemannian distances between $W_{i}$ and $L_{k}$ \cite[Lemma 6.9.4]{Jost:2017aa}. In view of Remark \ref{rem:Exp-map}, this interpretation is only approximately true mathematically, but still correct informally: $S_{i}(W)$ moves $W_{i}$ towards the normalized geometric mean of the likelihood vectors $L_{k},\,k \in \mc{N}_{i}$.

The similarity matrix induces the \textit{assignment flow} through a system of spatially coupled nonlinear ODEs which evolves the assignment vectors
\begin{subequations}\label{eq:assignment-flow}
\begin{align}
  \dot W &= R_{W}S(W),\quad W(0)=\eins_{\mc{W}}, \label{eq:assignment-flow-a} \\
  \dot W_{i} &= R_{W_{i}} S_{i}(W),\quad W_{i}(0)=\eins_{\mc{S}_n}.\quad i \in \mc{V},\label{eq:assignment-flow-b}
\end{align}
\end{subequations}
Integrating this flow numerically \cite{Zeilmann:2018aa} yields curves $W_{i}(t) \in \mc{S}_n$ for every pixel $i \in \mc{V}$ emanating from $W_{i}(0)=\eins_{\mc{S}_n}$, which approach some vertex (unit vector) of $\ol{\mc{S}_n}= \Delta_n$ and hence a unique label assignment after a trivial rounding $W_{i}(t)$ for sufficiently large $t > 0$.

\subsection{Linear Assignment Flow}
The \textit{linear assignment flow}, introduced by \cite{Zeilmann:2018aa}, uses the exponential map with respect to the e-connection \eqref{eq:Exp-e} in order to approximate the mapping \eqref{eq:def-SW} as part of the assignment flow \eqref{eq:assignment-flow-a} by
\begin{subequations}\label{eq:def-laf}
\begin{align}
\dot W &= R_{W}\Big[S(W_0)+dS(W_0)\big[\Exp_{W_0}^{-1}(W)\big]\Big], \\
W_{0}&=W(0)=\eins_{\mc{W}} \in \mc{W}.
\end{align}
\end{subequations}
This linear assignment flow \eqref{eq:def-laf} is still \textit{nonlinear} but admits the following  parametrization \cite[Prop.~4.2]{Zeilmann:2018aa}
\begin{subequations}\label{eq:laf}
\begin{align}
W(t) &= \Exp_{W_{0}}\big(V(t)\big),\\
\dot V(t) &= R_{W_{0}}\big[S(W_0) + dS(W_0)[V(t)]\big],\\
 V(0)&=0,
\end{align}
\end{subequations}
where the latter ODE is \textit{linear} and defined on the vector space $\mc{T}_{\mc{W}}$. Fixing $S(W_0)$ in the following, \eqref{eq:laf} is \textit{linear} with respect to both the tangent vector $V$ and the parameters $\omega_{ik}$ in the differential $dS(W_0)$ (see \eqref{eq:dSW0} and Remark~\ref{rem:alt-laf} below), that makes this approach attractive for parameter estimation.

It can be shown that $S_i(W)$ from \eqref{eq:def-SW} can equivalently be expressed with $\exp_{\BS{n}}$ as
\begin{equation}
  S_i(W) = \exp_{\BS{n}}\Big(\sum_{k \in \mc{N}_{i}} \omega_{ik} \Big(\exp^{-1}_{\BS{n}}(W_{k}) - \frac{1}{\rho} D_k\Big)\Big) 
\end{equation}
for all $i \in \mc{V},\; W\in \mc{W}$.
A standard calculation shows that the $i$-th component of the differential\\ $dS(W) \colon \mc{T}_{\mc{W}} \to \mc{T}_{\mc{W}}$ is given by
\begin{equation}\label{eq:dSW0}
\begin{split}
 &dS_i(W) \colon \mc{T}_{\mc{W}} \to T_n,\\
&dS_i(W)[V] = \sum_{k \in \mc{N}_i} \omega_{ik} \RO_{S_i(W)} \left[ \frac{V_k}{W_k} \right]
  \end{split}
\end{equation}
for all $V\in \mc{T}_0,\; i \in \mc{V}$.

\subsection{Numerical Integration of the Flow}\label{subsec:numeric-integration-Amfld}
Setting $\Lambda(V, W) := \exp_{W}(V)$ gives an action $\Lambda\colon \mc{T}_{\mc{W}} \times \mc{W} \to \mc{W}$ of the vector space $\mc{T}_{\mc{W}}$ viewed as an additive group on the assignment manifold $\mc{W}$. In \cite{Zeilmann:2018aa} this action is used to numerically integrate the assignment flow by applying geometric Runge-Kutta methods. The resulting method for an  arbitrary vector field $F \colon \mc{W} \to \mc{T}_{\mc{W}}$ is as follows. Suppose, the ODE
\begin{equation}\label{eq:general_ODE_W}
  \dot{W}(t) = \RO_{W(t)} [F(W(t)],\quad W(0) = \BW
\end{equation}
on the assignment manifold is given. Then the pa\-ra\-metrization $W(t) = \exp_{\BW}(V(t))$ yields an equivalent reparametrized ODE
\begin{subequations}\label{eq:general_ODE_V}
\begin{align}
  \dot{V}(t) &= F(W(t)) = F\big(\exp_{\BW}(V(t)\big), \\ 
  V(0) &= 0
\end{align}
\end{subequations}
purely evolving on the vector space $\mc{T}_{\mc{W}}$, where standard Runge-Kutta methods (cf.~Section~\ref{sec:Background}) can now be used for numerical integration. Translating these update schemes back onto $\mc{W}$, yields geometric Runge-Kutta methods on $\mc{W}$ induced by the Lie-group action $\Lambda = \exp$.
\begin{remark}
  Notice, that the assumption $F(W) \in \mc{T}_{\mc{W}}$ is crucial because the transformation of the ODE \eqref{eq:general_ODE_W} onto $\mc{T}_{\mc{W}}$ in \eqref{eq:general_ODE_V} uses the inverse of $\RO_{W}$, which only exists for elements of $\mc{T}_{\mc{W}}$ but not for $\R^{m\times n}$.
  However, this is no limitation.\\ Suppose any vector field $\tilde{F} \colon \mc{W} \to \R^{m\times n}$ is given. Due to $\RO_{W} = \RO_{W}\circ \PT$ by \eqref{eq:Rp-Pi0}, we may consider $F(W) := \PT[\tilde{F}(W)] \in \mc{T}_{\mc{W}}$ instead, without changing the underlying ODE \eqref{eq:general_ODE_W} for $W(t)$.
\end{remark}

In the following, we mainly use the Euler method to numerically integrate the flow \eqref{eq:general_ODE_V} on the vector space $\mc{T}_{\mc{W}}$, i.e.
\begin{equation}\label{eq:Euler-ODE-V}
\begin{split}
  V^{(k+1)} &= V^{(k)} + h_k F\big(W^{(k)}\big), \quad V^{(0)} = 0,\\
  W^{(k)} &= \exp_{\BW}\big(V^{(k)}\big)
  \end{split}
\end{equation}
with step-size $h_k > 0$. Due to the Lie-group action, this update scheme translates to the geometric Euler integration on $\mc{W}$ given by
\begin{subequations}\label{eq:geomEuler-ODE-W}
\begin{align}
  W^{(k+1)} &= \exp_{W^{(k)}}\Big(h_k F\big(W^{(k)}\big)\Big),\\
   W^{(0)} &= \BW,
\end{align}
\end{subequations}
with step-size $h_k > 0$.

\section{Learning Adaptive Regularization Parameters}
\label{sec:Adaptive-Learning-for-Dynamical-Systems}

In this section, we study the parameter learning approach \eqref{eq:parameter-problem-modified-laf}, which is a specific instance of the general formulation \eqref{eq:parameter-problem-general}. 
The goal is to adapt the regularization of the linear assignment flow \eqref{eq:def-laf} on the fixed time horizon $[0, T]$ controlled by the weights \eqref{eq:def-weights}, in the following collectively denoted by $\Omega$, so as to preserve important image structure in a supervised manner. During learning, the image structure is prescribed by given ground truth labeling information $W^{\ast}$, where every row $W^\ast_i$ is some unit basis vector $e_{k_i}$ of $\R^n$ representing the ground truth label $l_{k_i}$ at node $i\in \mc{V}$. The adaptivity of the weights with respect to the desired image structure is measured by $\mc{C}$ in terms of the discrepancy  between ground truth $W^\ast$ and the labeling induced by $V(T) = V(T, \Omega)$ at fixed time $T$. The corresponding optimization problem reads
\begin{subequations}\label{eq:parameter-problem-modified-laf}
\begin{align}
\min_{\Omega \in \mc{P}}\quad & \mc{C}\big( V(T, \Omega) \big ) \label{eq:objective-modified-laf}\\
\text{s.t.}\quad & \dot V(t) =F(V(t), \Omega),\quad t \in [0,T], \label{eq:constrained-modified-laf}\\
&V(0) = 0,
\end{align}
\end{subequations}
with components\\[2mm]
\begin{tabular}{p{0.1\textwidth} p{0.9\textwidth}}
$\mathcal{P}$   & parameter manifold, representing the weights $\omega_{ik}$ from \eqref{eq:def-weights}, see Section~\ref{subsec:param-mdfl}.\\
$F(V, \Omega)$  & modified version of the linear assignment flow \eqref{eq:modified-laf}, see Section~\ref{subsec:modified-laf}.
\\
$\mc{C}$ & a objective function measuring the discrepancy to the ground truth, see Section~\ref{subsec:obj-func-C}.
\end{tabular}\\[2mm]
It is important to note that the dependency of $\mc{C}(V(T, \Omega))$ on the weights $\Omega$ is only \textit{implicitly} given through the solution $V(T) = V(T, \Omega)$ of \eqref{eq:constrained-modified-laf}. In Section~\ref{subsec:num-optimization} we therefore present a numerical first-order scheme for optimizing \eqref{eq:parameter-problem-modified-laf} where the gradient of $\mc{C}(V(T, \Omega))$ with respect to the parameter $\Omega$ is calculated using the sensitivity analysis from Section~\ref{sec:Background}.

\subsection{Parameter Manifold}\label{subsec:param-mdfl}
In the following, we define the parameter manifold representing the weights $\omega_{ik}$ from \eqref{eq:def-weights} associated to the neighborhood $\mc{N}_i$, $i\in\mc{V}$. Based on this parametrization, we can compute the differential $dS(W_0)$ and thus describe the linear assignment flow \eqref{eq:laf} on the tangent space by a corresponding expression in Lemma~\ref{lem:alt-laf} below.

To simplify the exposition, we assume that all neighborhoods $\mc{N}_i$ have the same size
\begin{equation}
  N := |\mc{N}_i| \quad \text{for all } i \in \mc{V}.
\end{equation}
Due to the constraints \eqref{eq:def-weights}, the weight vector $\Omega_i := (\omega_{i1}, \ldots, \omega_{iN})^\T$ can be viewed as a point in $\mc{S}_N$. Accordingly, we define the \textit{parameter manifold} 
\begin{equation}
  \mc{P} := \mc{S}_N \times \ldots \times \mc{S}_N\qquad (m = |\mc{V}|\; \text{times})
\end{equation}
as feasible set for learning the weights, 
which has the form of an assignment manifold and thus also has a Riemannian structure $(\mc{P}, g)$, given by the Fisher-Rao metric. We use the identification
\begin{equation}
\begin{split}
  &\mc{P} = \{ \Omega \in \R^{m \times N} \colon \Omega \eins_N = \eins_m \text{ and } \Omega_{ik} > 0 \text{for all } i\in[m], k \in [N]\}.
\end{split}
\end{equation}
Points $\Omega \in \mc{P}$ now represent the global choice of weights with $\Omega_i$ representing the weights $\omega_{ik}$ associated to the neighborhood $\mc{N}_i$ in \eqref{eq:def-weights}. The constant tangent space of $\mc{P}$ is denoted by $\mc{T}_{\mc{P}}$ and the corresponding orthogonal projection by 
\begin{equation}
\begin{split}
  &\PT_{\mc{P}} \colon \R^{m\times N} \to \mc{T}_{\mc{P}},\quad M \mapsto \PT_{\mc{P}}[M] = (\PT_N[M_1], \ldots, \PT_N[M_m])^\T.
   \end{split}
   \end{equation}

Next, we give a global expression for the differential $dS(W)$ which will simplify following formulas and calculations. For this, we define the \textit{averaging matrix} $A_\Omega \in \R^{m\times m}$ with weights $\Omega \in \mc{P}$ by
\begin{equation}\label{eq:def-A-Omega-ik}
  (A_\Omega)_{ik} := \delta_{k\in\mc{N}_i} \Omega_{ik} = \begin{cases}
                                                           \Omega_{ik} &\text{, for } k\in\mc{N}_i\\
                                                           0 &\text{, else,}
                                                         \end{cases}
\end{equation}
where $\delta_{k\in\mc{N}_i}$ is the Kronecker-Delta with value $1$ if $k \in\mc{N}_i$ and $0$ otherwise. We observe that the averaging matrix $A_\Omega$  \textit{linearly} depends on the weight parameters. 

Thus, $A_\Omega$ parametrizes averages depending on the corresponding weights $\Omega$ with respect to the underlying graph structure, given by the neighborhoods \eqref{eq:def-neighborhood-i}. For a matrix $M \in \R^{m \times n}$, the averages of the row vectors with weights $\Omega$ are then just given by the matrix multiplication $A_\Omega M$, with the $i$-th row vector given by
\begin{equation}
  (A_\Omega M)_i = \sum_{k\in\mc{N}_i} \omega_{ik} M_k.\quad \text{for all } i\in\mc{V}.
\end{equation}
For later use, we record the following formula for the adjoint of $A_\Omega$ as a linear map with respect to $\Omega$.
\begin{lemma}\label{lem:adjoint-A}
  If the averaging matrix is viewed as a linear map $A \colon \R^{m\times N} \to \R^{m\times m}$, $\Omega \mapsto A_\Omega$, then the adjoint map $A^\T \colon \R^{m\times m} \to \R^{m\times N}$, $B \mapsto A_B^\T$ is given by
\begin{equation}\label{eq:adjoint-A}
  \big(A_B^\T)_{ij} = B_{ij}\quad\text{for } i \in \mc{V}, j\in \mc{N}_{i}.
\end{equation}

\begin{proof}
  For arbitrary $B \in \R^{m\times m}$ and $\Omega \in \R^{m\times N}$, we obtain $\la A_\Omega, B\ra = \sum_{i, j\in\mc{V}} \delta_{j\in\mc{N}_i} \Omega_{ik} B_{ik} = \la \Omega, A_B^\T\ra$ due to \eqref{eq:def-A-Omega-ik}.
\end{proof}
\end{lemma}

Using $A_\Omega$ with $\Omega \in \mc{P}$, it follows from \eqref{eq:dSW0} that $dS(W)$ can be expressed as
\begin{equation}\label{eq:global-dSW}
  dS(W)[V] = \RO_{S(W)} \left[A_\Omega\left(\frac{V}{W}\right)\right],
\end{equation}
for all $V \in \mc{T}_{\mc{W}},\; W \in \mc{W}$. As a result, the linear assignment flow \eqref{eq:laf} on the vector space $\mc{T}_{\mc{W}}$ can be parametrized as follows.

\begin{lemma}\label{lem:alt-laf}
  Using the parametrization $\ol{V} := n V$, the linear assignment flow \eqref{eq:laf} takes the form
  \begin{subequations}\label{eq:alt-laf}
  \begin{align}
W(t) &= \exp_{\BW}(\ol{V}(t)),\\
\dot{\ol{V}}(t) &= \PT[ S(W_0)] + R_{S(W_0)}[A_\Omega \ol{V}(t)],\\
\ol{V}(0) &= 0.
\end{align}
 \end{subequations}
\begin{proof}
  At $p = \BS{n}$, the linear map \eqref{eq:Rp-Pi0} takes the form
  \begin{align*}
    \RO_\BS{n} &= \Diag(\BS{n}) - \BS{n}\BS{n}^\T = \frac{1}{n}\big(I - \BS{n}\eins^\T\big) = \frac{1}{n}\PT_n,
  \end{align*}
  where $I \in \R^{n\times n}$ denotes the identity matrix. Because of $W_0 = \BW$, $\RO_{W_0} = \frac{1}{n}\PT$ follows. Therefore, $V = \frac{1}{n} \ol{V} = \RO_{\BW}\ol{V}$ which directly yields
  \begin{align*}
    W &= \Exp_{\BW}(V) = \Exp_{\BW}(\RO_{\BW}[\ol{V}]) = \exp_{\BW}(\ol{V}).
  \end{align*}
  Using \eqref{eq:global-dSW} together with $\frac{V}{W_0} = n V = \ol{V}$, the linear assignment flow \eqref{eq:laf} takes the form
  \begin{align*}
    \dot{V}(t) &= R_{W_{0}}\left[S(W_0) + \RO_{S(W_0)} \left[A_\Omega\left(\frac{V(t)}{W_0}\right)\right]\right] = \frac{1}{n}\PT\big[S(W_0) + \RO_{S(W_0)} [A_\Omega \ol{V}(t)]\big].
  \end{align*}
  As a consequence of $\PT \RO_{S(W_0)} = \RO_{S(W_0)}$ by \eqref{eq:Rp-Pi0}, the right-hand side of \eqref{eq:alt-laf} follows after multiplying the equation by $n$ and using $n\dot{V} = \dot{\ol{V}}$.
\end{proof}
\end{lemma}

\begin{remark}\label{rem:alt-laf}
  To simplify notation, we will write $V$ for $\ol{V}$ below. Equation \eqref{eq:alt-laf} highlights the importance to fix $S(W_0)$ in order to obtain a model that is linear in both the state vector $V$ and the parameters $\Omega$.
\end{remark}

\subsection{Modified Linear Assignment Flow }\label{subsec:modified-laf}

We now return to our objective to estimate the weight parameters $\Omega \in \mc{P}$ controlling the linear assignment flow on the fixed time interval $[0, T]$, in the supervised scenario \eqref{eq:parameter-problem-modified-laf}. \textcolor{black}{In this formulation, the data  represented by the likelihood matrix \eqref{eq:def-LW} only influence the linear assignment flow \eqref{eq:def-laf}, or equivalently \eqref{eq:alt-laf}, through the constant similarity matrix $S(W_0)$ that comprises \textit{averaged} data information depending on the initial choice of the weights $\Omega_0$. However, since the \textit{initial} weights are in general \textit{not} adapted to any specific image structure, this can lead to a loss of desired structural information through $S(W_0)$ at the outset, that cannot be recovered afterwards. }

\textcolor{black}{To avoid this problem, we slightly modify the linear assignment flow in \eqref{eq:alt-laf} to obtain an explicit data term that \textit{does not depend on the choice of the initial} weights. This is done through replacing the constant term $S(W_0)$ by the lifted distances $L(W_0)$, which results in the \textit{modified linear assignment flow}
}
\begin{subequations}\label{eq:modified-laf}
\begin{align}
&W(t) = \exp_{W_0}\big(V(t)\big),\\
\dot V(t) &= \PT[L(W_0)] + R_{S(W_0)}[A_\Omega V(t)]=: F(V(t), \Omega),V(0)=0.
\end{align}
\end{subequations}
\begin{remark}
  We point out that, strictly speaking, that the similarity matrix $S(W_0)$ is involved in two ways, in the constant term of \eqref{eq:alt-laf} and in the expression $\RO_{S(W_0)}$ of the differential $dS(W_0)$ (cf.~\eqref{eq:global-dSW}). However, the effect of the latter with respect to the initial weights is negligible, and the former appearance only causes the above mentioned loss of initial data information. We note again that \eqref{eq:modified-laf} is \textit{linear} with respect to both the tangent vector $V$ and the parameters $\Omega$ only if $S(W_0)$ is kept constant.
\end{remark}

\begin{proposition}\label{prop:dF-and-dFT}
The differential of the map $F \colon \mc{T}_{\mc{W}} \times \mc{P} \to \mc{T}_{\mc{W}}$ on the right-hand side of \eqref{eq:modified-laf} with respect to the first and second argument are given by
\begin{subequations}
\begin{align}
\begin{split}
&d_VF(V, \Omega) \colon \mc{T}_{\mc{W}} \to \mc{T}_{\mc{W}},\\
&X \mapsto d_V F(V, \Omega)[X] = R_{S(W_0)}[ A_\Omega X],\\ 
\end{split}
\\[2ex]
\begin{split}
&d_\Omega F(V, \Omega) \colon \mc{T}_{\mc{P}} \to \mc{T}_{\mc{W}},\\
&\Psi \mapsto d_\Omega F(V, \Omega) [\Psi] = R_{S(W_0)}[A_{\Psi} V].
\end{split}
\end{align}
\end{subequations}
The corresponding adjoint mappings with respect to the standard Euclidean structure of $\R^{m\times n}$ are 
\begin{subequations}
\begin{align}
\begin{split}
&d_VF(V, \Omega)^\T \colon \mc{T}_{\mc{W}} \to \mc{T}_{\mc{W}}, \\
&X \mapsto d_V F(V, \Omega)^\T[X] = A_\Omega^\T R_{S(W_0)}[X],
\end{split}
\\[2ex]
\begin{split}
&d_\Omega F(V, \Omega)^\T \colon \mc{T}_{\mc{W}} \to \mc{T}_{\mc{P}}, \\
&X \mapsto d_\Omega F(V, \Omega)^\T[X] = \PT_{\mc{P}}\big[A_{(R_{S(W_0)}[X]) V^\T}^\T\big],
\end{split}
\end{align}
\end{subequations}
with the adjoint $A^\T_{(\cdot)}$ from Lemma~\ref{lem:adjoint-A}.
\begin{proof}
  Let $V, X \in \mc{T}_{\mc{W}}$ and set $\gamma(t) := V + tX \in \mc{T}_0$ for all $t \in \R$. Then
  \begin{align*}
    &d_VF(V, \Omega)[X] = \frac{d}{dt} F(\gamma(t), \Omega) \big|_{t = 0} = R_{S(W_0)} [A_\Omega \dot{\gamma}(0)] = R_{S(W_0)} [A_\Omega X].
  \end{align*}
  Similarly, for $\Omega \in \mc{P}$ and $\Psi \in \mc{T}_{\mc{P}}$, let $\eta(t) := \Omega + t \Psi \in \mc{P}$ be a curve with $t \in (-\varepsilon, \varepsilon)$ for sufficiently small $\varepsilon>0$. The linearity of the averaging operator $A_\Omega$ with respect to $\Omega$ gives
  \begin{align*}
    &d_\Omega F(V, \Omega)[X] = \frac{d}{dt} F(V, \eta(t)) \big|_{t = 0} = \frac{d}{dt} R_{S(W_0)} [A_{\eta(t)} V] \big|_{t = 0} = R_{S(W_0)} A_\Psi [ V].
  \end{align*}
  We now determine the adjoint differentials. Consider arbitrary  $X,Y \in \mc{T}_{\mc{W}}$ and note that the linear map $R_S(W_0)$ is symmetric, since every component map $R_{S_i(W_0)}$ is symmetric by \eqref{eq:def-Rp}. Thus,
  \begin{align*}
    \la d_V F(V, \Omega)[Y], X\ra &= \la R_{S(W_0)} \left[ A_\Omega Y\right], X\ra = \la Y, A_\Omega^\T R_{S(W_0)}[X]\ra 
  \end{align*}
  and therefore $d_VF(V, \Omega)^\T[X] = A_\Omega^\T R_{S(W_0)}[X]$.
  Now let arbitrary $\Psi \in \mc{T}_{\mc{P}}$ and $X \in \mc{T}_{\mc{P}}$ be given. Then
  \begin{align*}
    &\la d_\Omega F(V, \Omega)[\Psi], X\ra = \la R_{S(W_0)} \left[ A_\Psi V\right], X\ra = \la A_\Psi, (R_{S(W_0)}[X]) V^\T\ra \\
    &= \la \Psi, A_{(R_{S(W_0)}[X]) V^\T}^\T \ra = \la \Psi, \PT_{\mc{P}}\big[A_{(R_{S(W_0)}[X]) V^\T}^\T\big]\ra,
  \end{align*}
  which proves the expression for the corresponding adjoint.
\end{proof}
\end{proposition}

\subsection{Objective Function}\label{subsec:obj-func-C}
Let $W = \exp_{\BW}(V)\in\mc{W}$ be an assignment induced by $V \in \mc{T}_{\mc{W}}$. Accumulating the 
$\KL$-divergence between the ground truth $W^*_i$ and $W_i$ for every node $i\in \mc{V}$,
\begin{equation}
  \KL(W_i^*, W_i) = \sum_{j\in [n]} W^*_{ij}\log\left( \frac{W^\ast_{ij}}{W_{ij}} \right) = \la W^*_i, \log(W^*_i)\ra - \la W^*_i, \log(W_i)\ra,
\end{equation}
results in a measure of the global deviation between $W$ induced by $V$ and the ground truth $W^*$
\begin{equation}\label{eq:C-KL}
  \mc{C}(V) := \sum_{i\in\mc{V}} \KL(W^*_i, \exp_{\BS{n}}(V_i)) = \la W^{\ast}, \log(W^{\ast}) \ra - \la W^{\ast}, \log \big(\exp_{\eins_{\mc{W}}}(V)\big) \ra.
\end{equation}
\begin{remark}\label{rem:objFunc-implicit-dependence-Omega}
  It is important to note that $\mc{C}$ does \textit{not explicitly} depend on the weights  $\Omega \in\mc{P}$. In problem formulation \eqref{eq:objective-modified-laf}, this dependency  is only given \textit{implicitly} through the evaluation of $\mc{C}$ at $V(T, \Omega)$, where $V(t, \Omega)$ is the object depending on the parameter $\Omega$ as solution of the modified linear assignment flow \eqref{eq:modified-laf}.
\end{remark}

\begin{proposition}\label{prop:Egrad-C}
The \textit{Euclidean gradient} of objective \eqref{eq:C-KL} for fixed $W^*\in \mc{W}$ is given by
\begin{equation}\label{eq:gradient-C-KL}
\Egrad \mc{C}(V) = \exp_{\eins_{\mc{W}}}(V) - W^{\ast} \quad \text{for} \quad V \in \mc{T}_{\mc{W}}.
\end{equation}

\begin{proof}
Let $V\in \mc{T}_{\mc{W}}$. Note that for every $i\in\mc{V}$
\begin{align*}
  \la W_i^*, \log\big(\exp_{\BS{n}}(V_i) \big)\ra
  = \la W_i^*, V_i - \log(\la \eins,e^{V_i}\ra) \eins \ra = \la W_i^*, V_i\ra + \log(\la \eins,e^{V_i}\ra).
\end{align*}
Hence the $\KL$-divergence between $W^*_i$ and the induced assignment $W_i = \exp_{\BS{n}}(V_i)$ takes the form
\begin{align*}
  \KL\big(W_i^*, W_i\big) &= \la W_i^*, \log(W_i^*)\ra - \la W_i^*, V_i\ra + \log(\la \eins,e^{V_i}\ra)
\end{align*}
and results in the following expression for $\mc{C}$ from \eqref{eq:C-KL},
\begin{equation*}
  \mc{C}(V) = \la W^*, \log(W^*)\ra - \la W^*, V\ra + \sum_{i\in[m]} \log(\la \eins,e^{V_i}\ra).
\end{equation*}
Take $X \in \R^{m\times n}$ and set $\gamma(t) := V + tX$ for $t \in \R$. The above formula for $\mc{C}$  then implies
\begin{align*}
  \la \partial C(V), X\ra &= \frac{d}{dt} C(\gamma(t))\big|_{t=0} = - \la W^*, X\ra + \sum_{i\in[m]} \frac{1}{\la \eins, e^{V_i}\ra} \la e^{V_i}, X_i\ra = \la \exp_{\eins_\mc{W}}(V) - W^*, X\ra.
\end{align*}
Since $X \in \R^{m \times n}$ was arbitrary, expression \eqref{eq:gradient-C-KL} follows.
\end{proof}
\end{proposition}

\subsection{Numerical Optimization}\label{subsec:num-optimization}
With the above definitions of $\mc{C}$ and $F$, the optimization problem \eqref{eq:parameter-problem-modified-laf} for adapting the weights of the modified linear assignment flow \eqref{eq:modified-laf} takes the form
\begin{subequations}\label{eq:specific-parameter-problem-modified-laf}
\begin{align}
&\min_{\Omega \in \mc{P}}\; \sum_{i\in\mc{V}} \KL(W^*_i, W_i(T, \Omega))\\
&\;\; \text{s.t.}\nonumber\\
 &\quad \dot{V}(t) =\PT[L(W_0)] + R_{S(W_0)}[A_\Omega[V(t)]],\\
 &\quad V(0) = 0,\\
 &\quad W(T, \Omega) = \exp_{\BW}(V(T, \Omega)),
\end{align}
\end{subequations}
with $t \in [0,T]$.
Our strategy for \textit{parameter learning} is to follow the \textit{Riemannian gradient descent flow} on the parameter manifold induced by the potential
\begin{equation}
\begin{split}
 & \Phi \colon \mc{P} \to \R,\quad \Omega \mapsto \Phi(\Omega) := \sum_{i\in\mc{V}} \KL(W^*_i, W_i(T, \Omega)).
\end{split}
\end{equation}
Due to \eqref{eq:relation_Rgrad_and_Egrad-W}, this Riemannian gradient flow on $\mc{P}$ takes the form
\begin{subequations}\label{eq:Omega-flow}
\begin{align}
\begin{split}
  \dot \Omega(t) &= -\Rgrad_{\mc{P}} \Phi\big(\Omega(t)\big)= -\RO_{\Omega}\big[\Egrad \Phi\big(\Omega(t))\big) \big],
  \end{split}\\
  \Omega(0) &= \BP,\label{eq:Omega-flow-initial}
\end{align}
\end{subequations}
where $\RO_{\Omega}$ is given by \eqref{eq:R-W} on $\mc{P}$ and initial value \eqref{eq:Omega-flow-initial} represents an \textit{unbiased} initialization, i.e.~\textit{uniform} weights at every patch $\mc{N}_{i}$ at $i\in \mc{V}$.

We discretize \eqref{eq:Omega-flow} using the geometric explicit Euler scheme \eqref{eq:geomEuler-ODE-W} from Section~\ref{subsec:numeric-integration-Amfld} with constant step-size $h'>0$, which results in Algorithm~\ref{alg:1}.

\vspace{0.2cm}
\begin{algorithm}
  \KwData{Initial weights $\Omega^{(0)} = \BP$, objective function $\Phi(\Omega) = \mc{C}\big(V(T, \Omega)\big)$}
  \KwResult{Weight parameter estimates $\Omega^{\ast}$}
  \tcp{geometric Euler integration 
  }
  \For{$k = 0, \dotsc, K$}{
    compute $\Egrad \Phi(\Omega^{(k)})$ \tcp*{Algorithm~\ref{alg:2}}
    $\Omega^{(k+1)} = \exp_{\Omega^{(k)}}\big(-h' R_{\Omega^{(k)}} \big[\Egrad \Phi(\Omega^{(k)})\big]\big) $\;
  }
  \caption{Discretized \textit{Riemannian flow} \eqref{eq:Omega-flow}.}\label{alg:1}
\end{algorithm}

\vspace{0.5cm}
Algorithm~\ref{alg:1} calls Algorithm~\ref{alg:2} that we explain next. 
As pointed out in Remark~\ref{rem:objFunc-implicit-dependence-Omega}, the dependency of $\Phi(\Omega) = \mc{C}(V(T, \Omega))$ on $\Omega$ is only implicitly given through the solution $V(t, \Omega)$ of the modified linear assignment flow \eqref{eq:modified-laf}, evaluated at time $T$. According to \eqref{eq:adjoint-sensitivity-column-form}, the gradient of $\Phi$ decomposes as 
\begin{equation}\label{eq:partial-Phi}
  \Egrad \Phi(\Omega) = d_\Omega V(T, \Omega)^\T\big[\Egrad \mc{C}(V(T, \Omega))\big],
\end{equation}
where $d_\Omega V(T, \Omega)^\T$ is the sensitivity of the solution $V(T, \Omega)$ with respect to $\Omega$. Thus, the major task is to determine the sensitivity of $V(T, \Omega)$ in order to obtain the gradient $\Egrad \Phi(\Omega)$, which in turn drives the Riemannian gradient descent flow and adapts the weights $\Omega$. To this end, we choose the discretize-then-differentiate approach  \eqref{eq:discrete-sensitivity-parameter} -- recall  the commutative diagram of Fig.~\ref{fig:sensitivity-diagram} and relations summarized as Remark \ref{rem:commuting} -- with the explicit Euler method  and constant step-size $h>0$, which results in Algorithm~\ref{alg:2}.

\vspace{0.25cm}
\begin{algorithm}
  \KwData{Current weights $\Omega^{(k)}$}
  \KwResult{Objective value $\Phi(\Omega^{(k)}) = \mc{C}(V^{(N)}(\Omega^{(k)}))$, adjoint sensitivity $\partial \Phi(\Omega^{(k)})$}
  \tcp{forward Euler integration 
  }
  \For{$j = 0, \dots, N-1$}{
  $V^{(j+1)} = V^{(j)} + h F\big(V^{(j)}, \Omega^{(k)}\big)$\;}
  compute $\lambda^{(N)} = \Egrad \mc{C}(V^{(N)}(\Omega^{(k)}))$\;
  set $\Egrad \Phi(\Omega) = 0$\;
  \tcp{backward Euler integration 
  }
  \For{$j = N-1, \dots, 0$}{
      $\lambda^{(j)} = \lambda^{(j+1)} + h d_V F\big(V^{(j)}, \Omega^{(k)}\big)^\T \lambda^{(j+1)}$\;
      $\Egrad \Phi(\Omega) \mathrel{+}= h d_\Omega F\big(V^{(j-1)}, \Omega^{(k)}\big)^{\T}\lambda^{(j)}$
      \tcp*{summand of \eqref{eq:partial-Phi}
    }
  }
   \caption{Computation of the Euclidean gradient $\Egrad \Phi(\Omega^{(k)})$ \eqref{eq:partial-Phi}.}
   \label{alg:2}
\end{algorithm}

\section{Experiments}
\label{sec:Experiments}

In this section, we demonstrate and evaluate our approach. \textcolor{black}{We start in Section \ref{sec:binary-letters} with a scenario of 2 labels and images of binary letters. 
We show that an adaptive regularizer, which is trained on letters with \textit{vertical and horizontal} structures only, effectively labels \textit{curvilinear} letters. This result 
illustrates the \textit{adaptivity} of regularization by using \textit{non-uniform} weights that are predicted for \textit{novel unseen} image data.}

In Section \ref{sec:curvilinear}, we consider a scenario with 3 labels and curvilinear line structure, that has to be detected and labeled explicitly in noisy data. Just using uniform weights for regularization must fail. In addition to the noise, the actual image structure is randomly generated as well and defines a class of images. We demonstrate empirically that learning the weights to adapt within local neighborhoods from example data solves this problem.

In Section \ref{sec:label-transport}, we adopt a different viewpoint and focus on pattern formation, rather than on pattern detection and recovery. We demonstrate the modeling expressiveness of the assignment flow with respect to pattern formation. In fact, even when using the \textit{linear} assignment flow as in the present paper, label information can be flexibly transported across the image domain under certain conditions. The experiments just indicate what can be done, in principle, in order to stimulate future work. We return to this point in Section \ref{sec:Conclusion}.

\textcolor{black}{
Regarding parameter learning, all experiments were conducted using the Euler scheme of Section \ref{sec:adjoint-Euler} for solving the adjoint system. Section \ref{sec:adjoint-Heun} provides a slightly more advanced alternative. While the latter methods integrates more accurately, the resulting overall costs depend on further factors whose evaluation is beyond the scope of this paper. For an in-depth study of numerical schemes in connection with geometric integration of the assignment flow, we refer to \cite{Zeilmann:2018aa}.
}
%

\subsection{{Adaptive Regularization of Binary Letters}}\label{sec:binary-letters}
\textcolor{black}{In this experiment we consider binary images of letters. The goal is to label a given letter image into \textit{foreground} and \textit{background} regions. In Figure~\ref{fig:binary-letters} these labels are encoded by $\big \{\fbox{\phantom{\crule[red]{0.28cm}{0.28cm}}}, \fbox{\crule[black]{0.28cm}{0.28cm}}\big\} = \{\textit{background}, \textit{ foreground}\}$. First we apply our approach during a \textit{training phase} in order to learn weight adaptivity for letters consisting of \textit{vertical and horizontal} structures (Fig.~\ref{fig:binary-letters}(a)). Afterwards we evaluate the approach in a \textit{test phase} using letters consisting of \textit{curvilinear} structures (Fig.~\ref{fig:binary-letters}(g)).}

\newlength{\SizeLetters}
\setlength{\SizeLetters}{0.3\textwidth}
\begin{figure}
\begin{center}
    \begin{tabularx}{1\linewidth}{c ? c ? c}
    {\begin{tabularx}{0.33\linewidth}{c}
          \includegraphics[width=\SizeLetters]{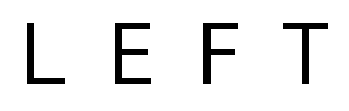} \\[-2mm]    
                \parbox{\SizeLetters}{\centering{(a) Training data}} \\[1mm]
                \includegraphics[width=1\SizeLetters]{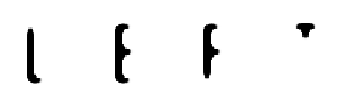}\\[-2mm]
                \parbox{\SizeLetters}{\centering{(b) labeling with \textit{uniform} weights}}\\[1mm]
                \includegraphics[width=\SizeLetters]{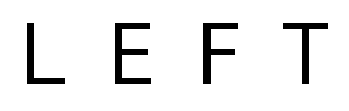} \\[-2mm]  
                \parbox{\SizeLetters}{\centering{(c) labeling with \textit{adaptive} weights}}\\[2mm]
                 \includegraphics[width=\SizeLetters]{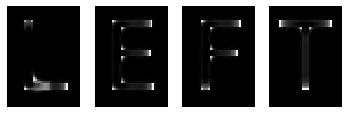} \\
\parbox{\SizeLetters}{\centering{(d) difference of adaptive\\[-1mm] to uniform \textit{weights}}}\\[3mm]
    \end{tabularx}}
    & 
    {\begin{tabularx}{0.25\linewidth}{p{4mm}p{4mm}p{1mm}p{4mm}p{4mm}} 
\includegraphics[height=1.6\SizeLetters]{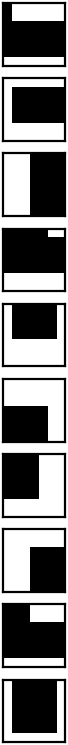} &
\includegraphics[height=1.6\SizeLetters]{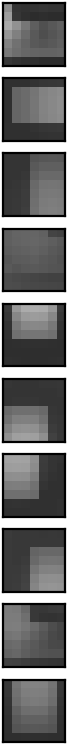} & 
&
\includegraphics[height=1.6\SizeLetters]{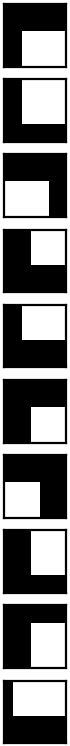} &
\includegraphics[height=1.6\SizeLetters]{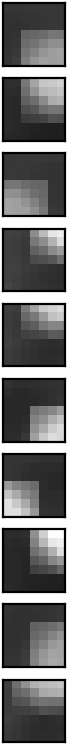}\\
\multicolumn{2}{c}{\centering(e)} &&\multicolumn{2}{c}{\centering(f)}\\[3mm]
\end{tabularx}}
&
    {\begin{tabularx}{0.33\linewidth}{c}
        \includegraphics[width=\SizeLetters]{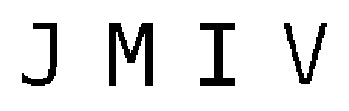} \\[-2mm]    
     \parbox{\SizeLetters}{\centering{(g) Test data}}\\[1mm]  
                  \includegraphics[width=\SizeLetters]{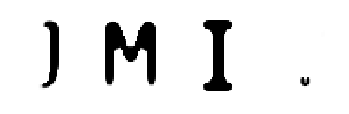}\\[-2mm]        
                         \parbox{\SizeLetters}{\centering{(h) labeling with \textit{uniform} weights}} \\[1mm]
                                      \includegraphics[width=\SizeLetters]{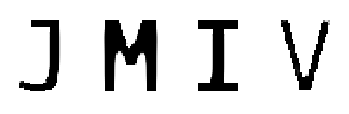} \\[-2mm]          
                                            \parbox{\SizeLetters}{\centering{(i) labeling with \textit{adaptive} weights}} \\[2mm]
                                                       \includegraphics[width=\SizeLetters]{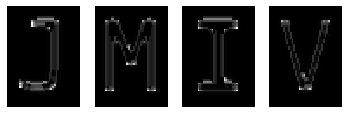}  \\       
                                                                 \parbox{\SizeLetters}{\centering{(j) difference of adaptive\\[-1mm] to uniform \textit{weights}}}\\[3mm]
        \end{tabularx}}
\\
{\normalsize \textbf{training data}}  & {\normalsize \textbf{coreset}} & {\normalsize \textbf{test data}}
    \end{tabularx}
\end{center}
    \caption{\textcolor{black}{\textbf{Adaptive regularization of binary letters.} \textsc{left column:} \textbf{(a)} Training data. \textbf{(b)} Labeling with uniform regularization fails. \textbf{(c)} Perfect adaptive reconstruction (sanity check). \textbf{(d)} illustrates weight adaptivity at each pixel in terms of the $\KL$-divergence of the \textit{learned} weight to the \textit{uniform} weight patch.
\textsc{middle column:} This column illustrates 10 pairs of image patches and corresponding weights for each class separately. The image patches with the corresponding optimal weight patches are illustrated by panel (e) for the \textit{foreground} class and by panel (f) for the \textit{background} class. We observe that the regularizer increases the influence of neighbors on the geometric averaging which belong to the respective class.
\textsc{right column:}
\textbf{(g)} Novel test data to be labeled using the regularizer trained on (a). \textbf{(h)} Uniform regularization fails \textbf{(i)} Adaptive reconstruction predicts \textit{curvilinear} structures of (e) using `knowledge' based on (a) where \textit{only vertical and horizontal} structures occur. \textbf{(j)} illustrates weights adaptivity at each pixel (cf. (d)).}
}
\label{fig:binary-letters}
  \end{figure}

\subsubsection{Training Phase}\label{sec:binary-letters-training}
\textcolor{black}{
Figure \ref{fig:binary-letters}(a) shows the binary images of letters which we used as \textit{training data}. Hereby, a given binary image served as input image and as ground truth as well. By using this data and solving problem \eqref{eq:specific-parameter-problem-modified-laf} we \textit{learn how to adapt the regularization parameter} of the modified linear assignment flow \eqref{eq:modified-laf}.}

\vspace{0.2cm}\textcolor{black}{
\textit{Optimization.} For each binary letter image, we solved problem \eqref{eq:specific-parameter-problem-modified-laf} using Algorithms \ref{alg:1} and \ref{alg:2} and the following parameter values: $|\mathcal{N}_{i}| = 7 \times 7$ (size of local neighborhoods, for every $i$), the Hamming distance (for the computation of the distance matrix \eqref{eq:def-distance-matrix}), $\rho = 0.5$ (scaling parameter for distance matrix, cf.~\eqref{eq:def-LW}), $h = 0.1$ (constant step-size for computing the gradient with Alg.~\ref{alg:2}), and $T=6$ (end of time horizon).
As for optimization on the parameter manifold $\mc{P}$ through the Riemannian gradient flow (Alg.~\ref{alg:1}), we used an initial value of $h' = 0.005$ together with backtracking for adapting the step-size. We terminated the iteration once the relative change 
\begin{equation}\label{eq:termination-learning}
 \frac{|\Phi ( \Omega^{(k)} ) - \Phi( \Omega^{(k-1)} )|}{h' |\Phi ( \Omega^{(k)} )|}
\end{equation}
of the objective function $\Phi\big( \Omega^{(k)} \big) = \mc{C}\big(V^{(N)}(\Omega^{(k)})\big)$ dropped below $0.01$ or the maximum number of $50$ iterations was reached.}

\vspace{0.2cm}\textcolor{black}{
\textit{Results.} 
The left column of Figure~\ref{fig:binary-letters} shows the results obtained during the training phase. Using \textit{uniform} weights fails completely to detect and label the letter structures (panel (b)). In contrast, the \textit{adapted} regularizer preserves the structure perfectly (panel (c)), i.e.~the \textit{optimal} weights steered the linear assignment flow towards the given ground-truth labeling. Panel (d) visualizes the weight \textit{adaptivity} at each pixel in terms of the $\KL$-divergence of the \textit{learned} weight to the \textit{uniform} weight patch.}

\subsubsection{Test Phase}\label{sec:binary-letters-testing}
\textcolor{black}{
During the training phase, optimal weights were associated with all training features through optimization, based on ground truth and a corresponding objective function. In the test phase with novel data and features, appropriate weights have to be \textit{predicted} because ground truth no longer is available. This was done by extracting a \textit{coreset} \cite{Phillips:2016aa} from the output generated by Algorithm \ref{alg:1} during the training phase, and constructing a map from novel features to weights, as described next.}

\vspace{0.2cm}
\textcolor{black}{
\textit{Coreset.} Let $\Omega^{\ast}$ denote the set of optimal weight patches generated by Algorithm \ref{alg:1}. As features we used $7 \times 7$ patches extracted from the training images. Let $P^{\ast}$ denote all \textit{feature vectors} $f_i,\,i\in\mc{V}$ (dimension $7 \times 7 = 49$) that were given as a point set in the Euclidean feature space $\mc{F} = \R^{49}$. We partitioned $P^{\ast}$ into two classes: \textit{foreground} and \textit{background}. Each class is represented by 156 prototypical patches extracted from the binary images. To each of these patches, a \textit{prototypical weight patch} was assigned, namely the geometric mean of all \textit{optimal} weight patches in $\Omega^{\ast}$ belonging to that patch.}

\textcolor{black}{
The middle column of Figure~\ref{fig:binary-letters} illustrates 10 pairs of image patches and corresponding weights for each class separately. By comparing the image patches with the corresponding optimal weight patches (cf. (e) \textit{foreground}, (f) \textit{background}), we observe that the regularizer increases the influence of neighbors on the geometric averaging which belong to the respective class.}

\vspace{0.2cm}
\textcolor{black}{
\textit{Mapping features to weights.} For a given \textit{novel} test image, we extracted $7 \times 7$ patches from the image, determined the closest image patch of the \textit{coreset} and assigned the corresponding weight patch to pixel $i$.
Note that we used the same size $7\times 7$ for the image patches and for the neighborhood size of geometric averaging.}

\vspace{0.2cm}\textcolor{black}{
\textit{Inference (labeling novel data).} In the test phase, we used the modified linear assignment flow and all parameter values in the same way, as was done during training. The only difference is that \textit{predicted} weight patches were used for regularization, as described above.}

\vspace{0.2cm}
\textcolor{black}{\textit{Results.} 
The right column of Figure~\ref{fig:binary-letters} shows the results obtained during the test phase. Panel (g) depicts the \textit{novel} (unseen) binary images. The next two panels show the labeling results using uniform weights (h) and using 
adaptive weights (i). Panel (j) illustrates the weight \textit{adaptivity} by showing the difference of predicted to uniform weights.}

\subsection{Adaptive Regularization of Curvilinear Line Structures}\label{sec:curvilinear}

We consider a collection of images containing line structures induced by random Voronoi diagrams (Fig.~\ref{fig:vessels}, panel (a)). The goal is pixel-accurate labeling of any given image with three labels representing: thin \textit{line} structure, \textit{homogeneous} region and \textit{texture}. In the figures below these labels are encoded by the three colors \big \{\crule[red]{0.3cm}{0.3cm}, \crule[green]{0.3cm}{0.3cm}, \crule[blue]{0.3cm}{0.3cm}\big\} = \{\textit{line}, \textit{homogeneous}, \textit{texture}\}. As usual in \textit{supervised machine learning}, our approach is first applied during a \textit{training phase} in order to learn weight adaptivity from ground truth labelings, and subsequently evaluated in a \textit{test phase} using \textit{novel unseen data}. 

\vspace{3mm}
\noindent
\subsubsection{Training Phase}
We used 20 randomly generated images together with ground truth as \textit{training data}: Figure \ref{fig:vessels}(a) shows one of these images and Figure \ref{fig:vessels}(b) the corresponding ground truth. \textcolor{black}{By following the same procedure as in Section~\ref{sec:binary-letters-training} we use these data in order to \textit{adapt the regularization parameter} of the modified linear assignment flow \eqref{eq:modified-laf} by solving problem
\eqref{eq:parameter-problem-modified-laf}, with the specific form given by \eqref{eq:specific-parameter-problem-modified-laf}.}

\newlength{\ImgSizeVessels}
\setlength{\ImgSizeVessels}{0.3\textwidth}
\begin{figure}
\begin{center}
    \begin{tabularx}{1\linewidth}{c c c}
	\includegraphics[width=1\ImgSizeVessels]{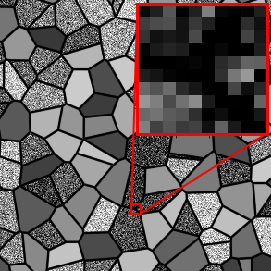} &	
	\includegraphics[width=1\ImgSizeVessels]{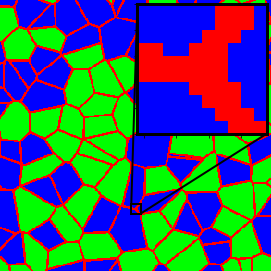} &
	\includegraphics[width=1\ImgSizeVessels]{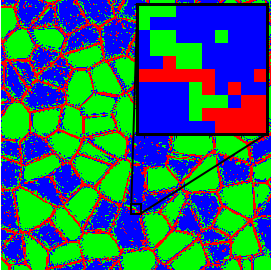}\\[2mm]
	\parbox{1\ImgSizeVessels}{\centering \bfseries (a) input scene} &
	\parbox{1\ImgSizeVessels}{\centering \bfseries (b) ground truth} & 
	\parbox{1\ImgSizeVessels}{\centering \bfseries (c) local rounding of distances \eqref{eq:D_ic}}
    \end{tabularx}
    \end{center}

\caption{\textcolor{black}{\textbf{Training data and local label assignments.}
The training data consist of 20 pairs of randomly generated images: \textbf{(a)} shows a randomly generated \textit{input image} from which features are extracted, as described in the text, and \textbf{(b)} the corresponding \textit{ground truth}. The ground truth images encode the labels 
 with colors \big \{\crule[red]{0.3cm}{0.3cm}, \crule[green]{0.3cm}{0.3cm}, \crule[blue]{0.3cm}{0.3cm}\big\} = \{line, homogeneous, texture\}.
Even though the global image structure can be easily assessed by the human eye, assigning correct labels pixelwise by an algorithm requires context-sensitive decisions, as the close-up view illustrates. \textbf{(c)} illustrates the quality of the distances \eqref{eq:D_ic} between extracted feature vectors. The panel shows the labeling obtained by local rounding, i.e.~by assigning to each pixel the label minimizing the corresponding distance. Comparing the close-up views of panel (b) and (c) shows that label assignments to individual pixels are noisy and incomplete. 
}}
 \label{fig:vessels}
\end{figure}

\vspace{0.2cm}
\textit{Feature Vectors.} 
The basis of our feature vectors are the outputs of simple $7 \times 7$ first- and second-order derivative filters, which are tuned to orientations at $0, 15,$\\$ \dots, 180$ degrees (we took absolute values of filter outputs to eliminate the $180 \sim 360$ degree symmetry). We reduced the dimension of the resulting feature vectors from $24$ to $12$ by taking the maximum of the first-order and second-order filter outputs, for each orientation. To incorporate more spatial information, we extracted $3 \times 3$ patches from this $12$-dimensional feature vector field. Thus, our  \textit{feature vectors} $f_i,\,i\in\mc{V}$ had dimension $3 \times 3 \times 12 = 108$ and were given as a point set in the Euclidean feature space $\mc{F} = \R^{108}$.
 
\vspace{0.2cm}
\textit{Label Extraction.}
Using ground truth information, we divided all feature vectors extracted from the training data into three classes: thin \textit{line} structure, \textit{homogeneous} region and \textit{texture}. We computed 200 prototypical feature vectors $l_{j c} \in \mc{F}$, $j\in[200]$, in each class $c\in\{\text{line}, \text{homogeneous}, \text{texture}\}$ by $k$-means clustering. Thus, each \textit{label} (line, homogeneous, texture) was represented by 200 feature vectors in $\mc{F}$.

\vspace{0.2cm}
\textit{Distance Matrix.}
Even though in the original formulation \eqref{eq:def-GJ} labels are represented by a single feature vector, multiple representatives can be taken into account as well by modifying the distance matrix \eqref{eq:def-distance-matrix} accordingly. With the identification
\[c \in \{\text{line}, \text{homogeneous}, \text{texture}\} = \{1, 2, 3\},\]
we defined the entries of the distance matrix $D_{ic}$, for every $i \in\mc{V}$, as the distance between $f_i$ and the best fitting representative $l_{jc}$ for class $c$, i.e.
\begin{equation}\label{eq:D_ic}
  D_{ic} := \min_{j\in[200]} \|f_i - l_{jc}\|_2.
\end{equation}
The quality of this distance information is illustrated by Figure \ref{fig:vessels}(c) that shows the labeling obtained by \textit{local rounding}, i.e.~by assigning to each pixel $i$ the label $c = \min_{\tilde c} D_{i\tilde c}$. Although the result looks similar to the ground truth \textcolor{black}{(cf. Figure \ref{fig:vessels}(b))}, it is actually quite noisy when looking to single pixels in the close-up view of Figure \ref{fig:vessels}(c).

\vspace{0.2cm}
\textit{Optimization.} For each input image of the training set, we solved problem \eqref{eq:parameter-problem-modified-laf} using Algorithms \ref{alg:1} and \ref{alg:2} and the following parameter values: $|\mathcal{N}_{i}| = 9 \times 9$ (size of local neighborhoods, for every $i$), $\rho = 1$ (scaling parameter for distance matrix, cf.~\eqref{eq:def-LW}), $h = 0.5$ (constant step-size for computing the gradient with Alg.~\ref{alg:2}), and $T=6$ (end of time horizon).
As for optimization on the parameter manifold $\mc{P}$ through the Riemannian gradient flow (Alg.~\ref{alg:1}), we used an initial value of $h' = 0.0125$ together with backtracking for adapting the step-size. We terminated the iteration \textcolor{black}{once the relative change \eqref{eq:termination-learning} of the objective function dropped below $0.001$ or the maximum number of $100$ iterations was reached.}

\newlength{\ImgSizeTrainingData}
\setlength{\ImgSizeTrainingData}{0.22\textwidth}
\begin{figure}[h!]
\begin{center}
    \begin{tabularx}{0.98\linewidth}{c c c c }
    	\includegraphics[width=\ImgSizeTrainingData]{./vessels_zoomed_scene009} &
	\includegraphics[width=\ImgSizeTrainingData]{./vessels_zoomed_rounded_distance_scene009}&
	\includegraphics[width=\ImgSizeTrainingData]{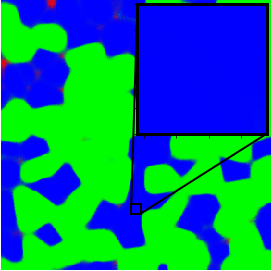} &
	\includegraphics[width=\ImgSizeTrainingData]{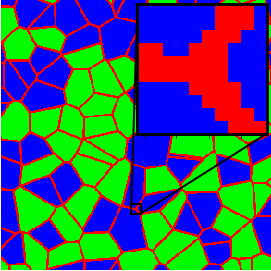} \\
	\includegraphics[width=\ImgSizeTrainingData]{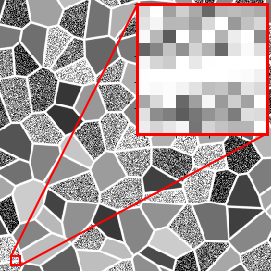} &
	\includegraphics[width=\ImgSizeTrainingData]{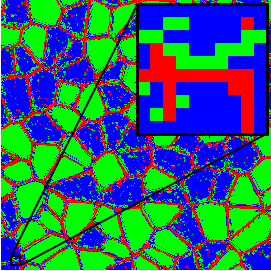}&
	\includegraphics[width=\ImgSizeTrainingData]{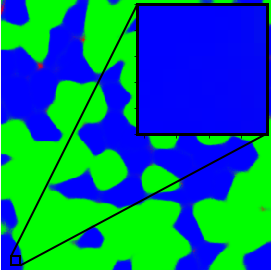} &
	\includegraphics[width=\ImgSizeTrainingData]{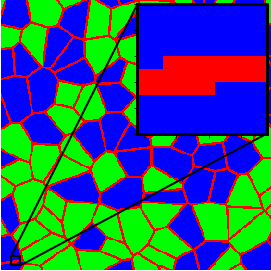} \\
\parbox{\ImgSizeTrainingData}{\centering  \bfseries (a) input scene} &  
\parbox{\ImgSizeTrainingData}{\centering  \bfseries (b) \textit{local rounding} of\\ distance information} &  
\parbox{\ImgSizeTrainingData}{\centering  \bfseries (c) labeling with\\ \textit{uniform} weights} &  
\parbox{\ImgSizeTrainingData}{\centering  \bfseries (d) labeling with\\ \textit{optimal} weights}
    \end{tabularx}
    \end{center}

\caption{\textbf{Training phase: Labeling results.}
This figure shows results of the training phase. Panel \textbf{(a)} shows the given input scene and panel \textbf{(b)} the corresponding \textit{locally rounded} distance information. The labeling with 
\textit{uniform} regularization (panel \textbf{(c)}) returns \textit{smoothed over} regions and completely fails  to preserve the line structures. The \textit{adaptive} regularizer preserves the line structure perfectly (panel \textbf{(d)}), i.e.~the \textit{optimal weights} are able to steer the linear assignment flow successfully towards the given ground-truth labeling. (\{\crule[red]{0.3cm}{0.3cm}, \crule[green]{0.3cm}{0.3cm}, \crule[blue]{0.3cm}{0.3cm}\} = \{line, homogeneous, texture\})}
\label{fig:results-training}
\end{figure}

 \vspace{0.2cm}
\textit{Results.} 
Figure \ref{fig:results-training} shows two results obtained during the training phase. They illustrate \textit{non-adaptive} regularization using \textit{uniform} weights, which results in blurred partitions and fails completely to detect and label the line structures (panel (c)). On the other hand, the \textit{adapted} regularizer preserves and restores the structure perfectly (panel (d)), i.e.~the \textit{optimal} weights steered the linear assignment flow towards the given ground-truth labeling.

\newlength{\SizeVessel}
\setlength{\SizeVessel}{0.2\textwidth}
\begin{figure}[h!]
\begin{center}
    \begin{tabularx}{0.9\linewidth}{c c c c}
     \includegraphics[width=\SizeVessel]{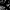}  
     &\includegraphics[width=\SizeVessel]{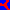}  
     &\includegraphics[width=\SizeVessel]{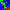}  
          &\includegraphics[width=\SizeVessel]{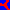}  
      \\
     \parbox{\SizeVessel}{\centering{\bfseries (a) training data}} 
     &\parbox{\SizeVessel}{\centering{\bfseries (b) ground truth}} 
     &\parbox{\SizeVessel}{\centering{ \bfseries (c) \textit{local rounding} of\\[-1mm] distance information}} 
          &\parbox{\SizeVessel}{\centering{ \bfseries (d) labeling with\\[-1mm] \textit{optimal} weights}} 
     \\[4mm]
    \end{tabularx}
     \includegraphics[width=3\SizeVessel]{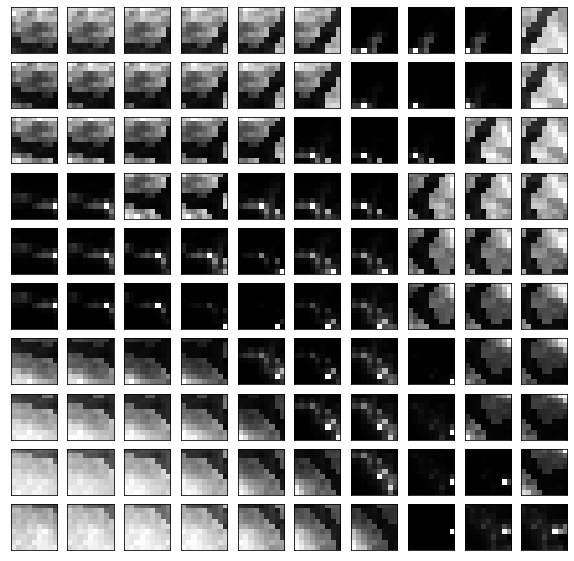}  \\     
     \parbox{4\SizeVessel}{\centering{ \bfseries (e) Optimal weight patches}} 
\end{center}
\caption{
\textbf{Training phase: Optimal weight patches.} \textsc{Top row:} \textbf{(a)} Close-up view of training data ($10 \times 10$ pixel region). \textbf{(b)} The corresponding ground truth 
section. \textbf{(c)} Local label assignments. \textbf{(d)} Correct labeling using adapted optimal  weights. \textsc{Bottom row:} \textbf{(e)} The corresponding optimal weight patches 
($10 \times 10$ grid), one patch for each pixel. Close to the line structure, the regularizer increases the influence of neighbors on the geometric averaging of assignments whose distances match the prescribed ground truth labels. Away from the line structure, the regularizer has learned to \textit{suppress} with small weights neighbors belonging to a line structure.
}
\label{fig:results-local-learned-vessel}
  \end{figure}

Figure \ref{fig:results-local-learned-vessel} shows a close-up view of a $10\times 10$ pixel region together with the corresponding $10\times 10$ \textit{optimal weight patches}, extracted from $\Omega^*$. The top row depicts {(a)} the training data, {(b)} the corresponding ground truth, {(c)} the local label assignments, and {(d)} the labeling obtained when using the learned weights $\Omega^*$. Plot {(e)} shows the corresponding optimal weight patches $\Omega_{i}^* = (\omega_{i1}, \ldots, \omega_{i\mc{N}})^\T$ associated to every pixel $i$ in the $10\times 10$ pixel region, where small and large weights are indicated by dark and bright gray values, respectively. These weight patches illustrate the result of the \textit{learning  process} for adapting the weights. Close to the line structure, the regularizer \textit{increases} the influence (with larger weights) of neighbors whose distance information matches the prescribed ground truth label. Away from the line structure, the regularizer has learned to \textit{suppress} (with small weights) neighbors that belong to a line structure.

\subsubsection{Test Phase}
\textcolor{black}{
As already explained in Section~\ref{sec:binary-letters-testing}, we have to \textit{predict} appropriate weights for novel data and features. We proceed as done before by extracting a \textit{coreset} from the output generated by Algorithm \ref{alg:1} and constructing a map from novel features to weights.}

\vspace{0.2cm}
\textit{Coreset.} Let $\Omega^{\ast}$ denote the set of optimal weight patches generated by Algorithm \ref{alg:1}, and let $P^{\ast}$ denote the set of all $15 \times 15$ patches of local label assignments based on the corresponding training features and distance \eqref{eq:D_ic}. We partitioned $P^{\ast}$ into three classes: thin \textit{line structures}, \textit{homogeneous} regions and \textit{texture}, and extracted for each class separately 225 prototypical patches by $k$-means clustering. To each of these patches and the corresponding cluster, a \textit{prototypical weight patch} was assigned, namely the weighted geometric mean of all \textit{optimal} weight patches in $\Omega^{\ast}$ belonging to that cluster. As weights for the averaging we used the Euclidean distance between the respective patches of local label assignments and the corresponding cluster centroid.

Figure \ref{fig:coreset-vessels} depicts 10 pairs of patches of prototypical label assignments and weights, for each of the three classes: line, homogenous, texture. Comparing these weight patches with the optimal patches depicted by Figure \ref{fig:results-local-learned-vessel}, we observe that the former are regularized (smoothed) by geometric averaging and, in this sense, summarize and represent all optimal weights computed during the training phase.

\vspace{0.2cm}
\textit{Mapping features to weights.} For each \textit{novel} test image, we extracted features using the same procedure as done in the \textit{training phase} and computed at each pixel $i$ the patch of local label assignments. For the latter patch, the closest patch of local label assignments of the \textit{coreset} was determined, and the corresponding weight patch was assigned to pixel $i$.

Note that the patch size $15\times 15$ of local label assignments was chosen larger as the patch size $9\times 9$ of the weights that was used both during training and for testing. The former larger neighborhood defines the local `feature context' that is used to predict weights for novel data.

\vspace{0.2cm}
\textit{Inference (labeling novel data).} In the test phase, we used the modified linear assignment flow and all parameter values in the same way, as was done during training. The only difference is that \textit{predicted} weight patches were used for regularization, as described above.

\vspace{0.2cm}
\textit{Results.}
 Figure \ref{fig:results-testing} shows a result of the test phase. Since all data are \textit{randomly} generated, this result is representative for the entire image class. 
\textcolor{black}{The panels (a) and (g) show the input data}, whereas ground truth (b) is only shown for visual comparison. \textcolor{black}{Panel (c) shows the labeling obtained using uniform weights and (d) illustrates the difference of (c) to the ground truth (b). Panel (e) shows the labeling obtained using adaptive weights, and (f) the corresponding difference of (e) to the ground truth (b). The labeling result clearly demonstrated the impact of weight \textit{adaptivity}.} This aspect is further illustrated in panel (h).

Figure \ref{fig:results-local-testing} shows \textit{predicted} weight patches for novel test data in the same format as Figure \ref{fig:results-local-learned-vessel} depicts \textit{optimal} weight patches computed during training. The similarity of the behaviour of predicted and optimal weights for pixels close and away from local line structure, demonstrates that the approach generalizes well to novel data. Since these data are randomly generated, this performance is achieved for any image data in this class.

\newlength{\SizeCoreset}
\setlength{\SizeCoreset}{0.6\textwidth}
\begin{figure}[h!]
    \begin{tabularx}{0.62\linewidth}{c}
     \includegraphics[width=\SizeCoreset]{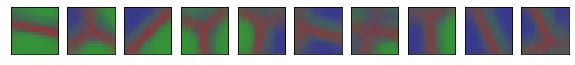} \\
      \includegraphics[width=\SizeCoreset]{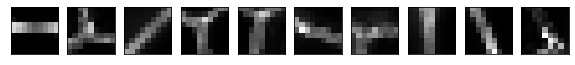} \\
     \parbox{\SizeCoreset}{\centering{ \bfseries (a) line}} \\[2mm]
               \includegraphics[width=\SizeCoreset]{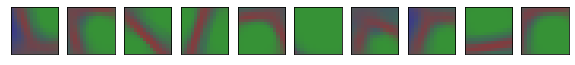} \\
           \includegraphics[width=\SizeCoreset]{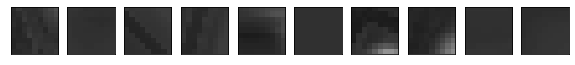} \\
     \parbox{\SizeCoreset}{\centering{ \bfseries (b) homogeneous}}\\[2mm]
               \includegraphics[width=\SizeCoreset]{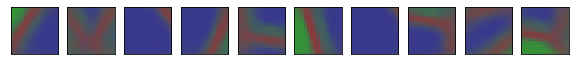} \\
           \includegraphics[width=\SizeCoreset]{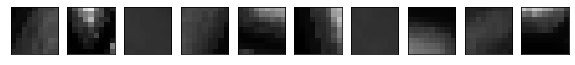}\\
     \parbox{\SizeCoreset}{\centering{ \bfseries (c) texture}}
    \end{tabularx}
\caption{
\textbf{Coreset visualization.} This plot shows $3\times 10$ prototypical patches of local label assignments and the corresponding weight patches of the coreset, for each of the 3 classes. \textbf{(a)} 10 prototypical pairs of the class \textit{line}. Weight patches `know' to which neighbors large weights have to be assigned, such that the local line structure is labeled correctly. 
\textbf{(b)} Weight patches of the \textit{homogeneous} label class are almost uniform, which is plausible, because the noisy assignments can be filtered most effectively. \textbf{(c)} The weight patches of the \textit{texture} label are comparable to the \textit{homogeneous} ones and almost uniform, for the same reason. 
(Color code \big \{\crule[red]{0.3cm}{0.3cm}, \crule[green]{0.3cm}{0.3cm}, \crule[blue]{0.3cm}{0.3cm}\big\} = \{line, homogeneous, texture\}).
}
\label{fig:coreset-vessels}
  \end{figure}

\newlength{\SizeVesselPrediction}
\setlength{\SizeVesselPrediction}{0.28\textwidth}
\begin{SCfigure}
\begin{tabularx}{0.6\linewidth}{c c}
	\includegraphics[width=\SizeVesselPrediction]{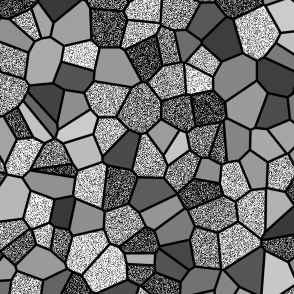} &
	\includegraphics[width=\SizeVesselPrediction]{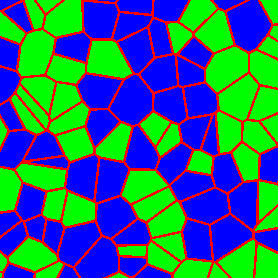} \\
	\parbox{\SizeVesselPrediction}{\centering{ \bfseries (a) novel (test) data}} &
	\parbox{\SizeVesselPrediction}{\centering{ \bfseries (b) ground truth}} \\[3mm]
	\includegraphics[width=\SizeVesselPrediction]{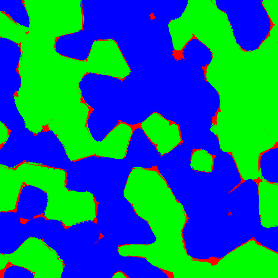} &
          \includegraphics[width=\SizeVesselPrediction]{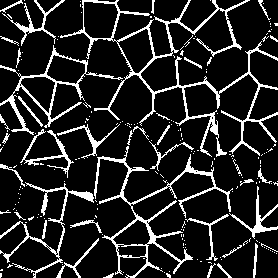} \\
	\parbox{\SizeVesselPrediction}{\centering{ \bfseries (c) labeling with \textit{uniform} weights}} &
	\parbox{\SizeVesselPrediction}{\centering{ \bfseries (d) difference of (c) to (b)}} \\[3mm]
          \includegraphics[width=\SizeVesselPrediction]{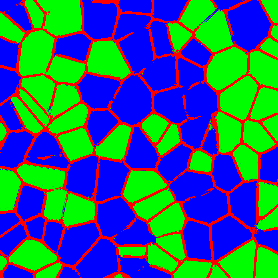} &
          \includegraphics[width=\SizeVesselPrediction]{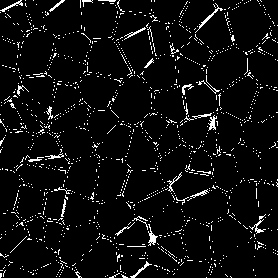} \\
          \parbox{\SizeVesselPrediction}{\centering{\bfseries  (e) labeling with \textit{adaptive} weights}} &
          \parbox{\SizeVesselPrediction}{\centering{ \bfseries (f) difference of (e) to (b)}}  \\[3mm]
	\includegraphics[width=\SizeVesselPrediction]{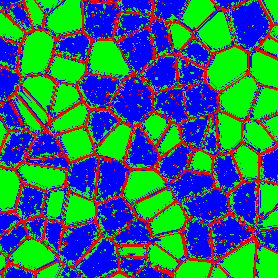} &
          \includegraphics[width=\SizeVesselPrediction]{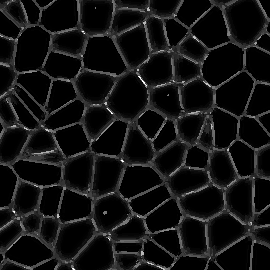}\\
	 \parbox{\SizeVesselPrediction}{\centering{ \bfseries (g) \textit{local rounding} of\\ distance information}} & 
	 \parbox{\SizeVesselPrediction}{\centering{ \bfseries (h) difference of adaptive\\[-1mm] to uniform \textit{weights}}}
    \end{tabularx}
        \caption{\textcolor{black}{\textbf{Test phase: Labeling results.} \textbf{(a)} Randomly generated novel input data, \textbf{(b)} the corresponding ground truth. \textbf{(c)} Labeling using \textit{uniform} weights fails to detect and label line structures. \textbf{(d)} illustrates the difference of (c) to the ground truth (b). \textbf{(e)} \textit{Adaptive} regularizer based on \textit{predicted} weights yields a result that largely agrees with ground truth. \textbf{(f)} shows the difference of (e) to the ground truth (b). 
       \textbf{(g)} shows the corresponding \textit{locally rounded} distance information extracted from the image data (a). Panel \textbf{(h)} illustrates weights adaptivity at each pixel in terms of the distance of the \textit{predicted} weight patch to the \textit{uniform} weight.
}}
\label{fig:results-testing}
  \end{SCfigure}

\begin{figure}[h!]
\begin{center}
    \begin{tabularx}{0.9\linewidth}{c c c c}
     \includegraphics[width=\SizeVessel]{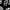}  
     &\includegraphics[width=\SizeVessel]{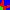}  
     &\includegraphics[width=\SizeVessel]{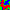}  
          &\includegraphics[width=\SizeVessel]{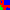}  \\
               \parbox{\SizeVessel}{\centering{ \bfseries (a) novel (test) data}} 
     &\parbox{\SizeVessel}{\centering{ \bfseries (b) ground truth}} 
     &\parbox{\SizeVessel}{\centering{ \bfseries (c) rounded distance}} 
          &\parbox{\SizeVessel}{\centering{ \bfseries (d) adaptive labeling}} 
      \\[4mm]
    \end{tabularx}
         \includegraphics[width=3\SizeVessel]{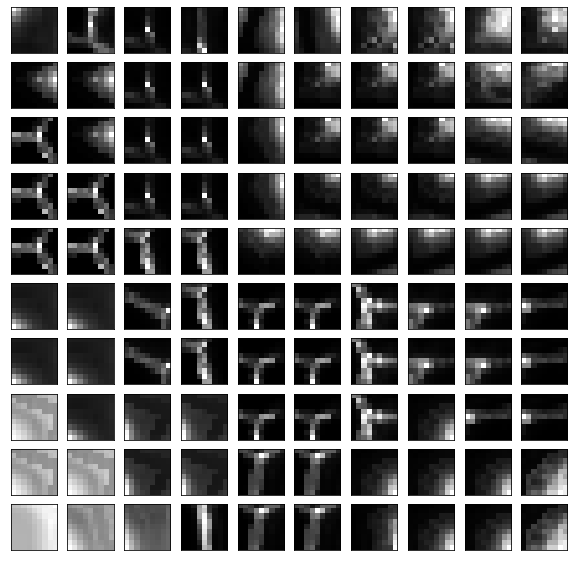}  
         \parbox{4\SizeVessel}{\centering{ \bfseries (e) predicted weight patches}}
\end{center}
\caption{
\textbf{Test phase: Predicted weight patches.} \textsc{Top row:} \textbf{(a)} Close-up view of \textit{novel} data ($10 \times 10$ pixel window). \textbf{(b)} Corresponding ground 
truth section (just for visual comparison, \textit{not} used in the experiment). \textbf{(c)} Local label assignment. \textbf{(d)} Labeling result using adaptive regularization with predicted weights. \textsc{Bottom:} \textbf{(e)} Corresponding predicted weight patches 
($10 \times 10$ grid), one patch for each pixel of the test data (a). The \textit{predicted} weight patches behave similar to the \textit{optimal} weight patches depicted by Fig.~\ref{fig:results-local-learned-vessel}, that were computed during the training phase (for different data). This shows that our approach generalizes to novel data.
}
\label{fig:results-local-testing}
  \end{figure}

\subsection{Pattern Formation by Label Transport}\label{sec:label-transport}
In this section, we illustrate the \textit{model expressiveness} of the assignment flow. \textcolor{black}{Specifically, we choose an input image and a target labeling which patterns are \textit{quite different}. The task is to estimate weights in order to steer the assignment flow to the target labeling. We show that our learning approach can determine the weights that `connect' these patterns by the assignment flow.} This shows that the weights which determine the regularization properties of the assignment flow actually encode information for \textit{pattern formation}. Finally, we briefly point out and illustrate in Section \ref{sec:Parameter-Learning-vs-OC} limitations of the current version of our approach.

\textcolor{black}{For the patterns below we used $\mathcal{X} = \big \{\fbox{\phantom{\crule[red]{0.28cm}{0.28cm}}}, \fbox{\crule[black]{0.28cm}{0.28cm}}\big\} = \{\textit{background}, \textit{ foreground}\}$ as labels and the Hamming distance for the computation of the distance matrix \eqref{eq:def-distance-matrix}.}

\subsubsection{Pattern Completion}
The \textit{top row} of Figure \ref{fig:fern-labeling-results} shows \textcolor{black}{the \textit{input image} and 
the \textit{target labeling}. The second row illustrates the evolution of the \textit{linear} assignment flow using optimal weight parameters. These optimal parameters were obtained by the Riemannian gradient flow on the parameter manifold in order to solve problem \eqref{eq:specific-parameter-problem-modified-laf}, which effectively steers the assignment flow to the target labeling.}
 
Having obtained the optimal weights $\Omega^{\ast}$ after convergence, we inserted them into the original \textit{nonlinear} assignment flow. The evolution of corresponding label assignments is shown by the third row of Figure \ref{fig:fern-labeling-results}. The fact that the label assignment at the final time $T$ is close to the target labeling which the linear assignment flow reaches exactly, confirms the remarkably close approximation of the nonlinear flow by the linear assignment flow, as already demonstrated in \cite{Zeilmann:2018aa} in a completely different way.

The rightmost panel in the top row of Figure \ref{fig:fern-labeling-results} shows, for each pixel, the deviation of the optimal weight patch form uniform weights. While it is obvious that the `source labeling' of the input data receives large weights, the spatial arrangement of weights at all other locations is hard to predict beforehand by humans. This is why \textit{learning} them is necessary.

\newlength{\ImgSizeFeather}
\setlength{\ImgSizeFeather}{0.11\textwidth}
\begin{figure}[h!]
\begin{center}
    \begin{tabularx}{0.7\linewidth}{c c c c}
	\fbox{\includegraphics[width=1.5\ImgSizeFeather]{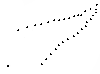}} &
	\fbox{\includegraphics[width=1.5\ImgSizeFeather]{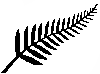}} &
	\fbox{\includegraphics[width=1.5\ImgSizeFeather]{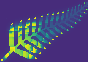}} &
	{\includegraphics[height=1\ImgSizeFeather]{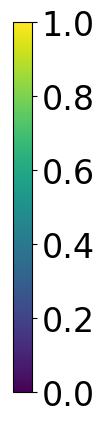}}\\
\parbox{1.5\ImgSizeFeather}{\centering \textcolor{black}{Input image}} &
\parbox{1.5\ImgSizeFeather}{\centering \textcolor{black}{Target labeling}} &
\parbox{1.75\ImgSizeFeather}{\centering \textcolor{black}{difference of adaptive\\[-1mm] to uniform \textit{weights}}}
    \end{tabularx}
    \end{center}
    
    \begin{center}
    \begin{tabularx}{1\linewidth}{c c c c c c c}            
    \rotatebox{90}{{\hspace{2mm} \textbf{Linear}}} &
        	\fbox{\includegraphics[width=\ImgSizeFeather]{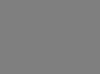}}&
	\fbox{\includegraphics[width=\ImgSizeFeather]{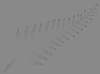}} &
	\fbox{\includegraphics[width=\ImgSizeFeather]{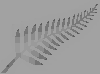}} &
	\fbox{\includegraphics[width=\ImgSizeFeather]{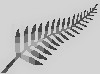}} &
	\fbox{\includegraphics[width=\ImgSizeFeather]{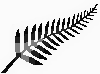}} &
	\fbox{\includegraphics[width=\ImgSizeFeather]{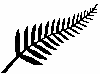}} \\
            \rotatebox{90}{{\hspace{-1mm} \textbf{Nonlinear}}} &
        	\fbox{\includegraphics[width=\ImgSizeFeather]{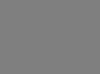}}&
	\fbox{\includegraphics[width=\ImgSizeFeather]{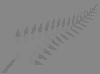}} &
	\fbox{\includegraphics[width=\ImgSizeFeather]{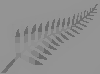}} &
	\fbox{\includegraphics[width=\ImgSizeFeather]{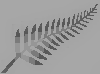}} &
	\fbox{\includegraphics[width=\ImgSizeFeather]{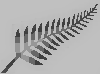}} &
	\fbox{\includegraphics[width=\ImgSizeFeather]{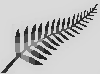}} \\
	&
\parbox{\ImgSizeFeather}{\centering\scriptsize $t=0$} & 
\parbox{\ImgSizeFeather}{\centering\scriptsize $t=1$} & 
\parbox{\ImgSizeFeather}{\centering\scriptsize $t=2$} &
\parbox{\ImgSizeFeather}{\centering\scriptsize $t=3$} &
\parbox{\ImgSizeFeather}{\centering\scriptsize $t=4$} &
\parbox{\ImgSizeFeather}{\centering\scriptsize $T=5$}
    \end{tabularx}
\end{center}
\caption{\textbf{Pattern completion.} This figure illustrates the model expressiveness of the assignment flow. \textsc{top row:} \textcolor{black}{Input image and target labeling}. The task was to estimate weights in order to steer the assignment flow to the target labeling. The rightmore panel illustrates, for each pixel, the distance of \textit{uniform} weights from the \textit{optimal} estimated weight patch.  
\textsc{middle row:} Label assignments of the linear assignment flow \textcolor{black}{using optimal weights obtained by solving \eqref{eq:parameter-problem-modified-laf}.} The Riemannian gradient flow on the parameter manifold effectively steers the flow to the target labeling. \textsc{bottom row:} Label assignments of the \textit{nonlinear} assignment flow using the optimal weights that were estimated using the \textit{linear} assignment flow. Closeness of both labeling patterns at the final point of time $T=5$ demonstrates that the linear assignment flow provides a good approximation of the full nonlinear flow. }
\label{fig:fern-labeling-results}
\end{figure}

\subsubsection{Transporting and Enlarging Label Assignments}
We repeated the experiment of the previous section using the academic scenario depicted by Figure \ref{fig:moving-mass-labeling-results}. A major difference is that locations of the input \textcolor{black}{image} \textit{do not} form a subset of the locations of the target labeling. As a consequence, the corresponding `mass' of assignments has to be both \textit{transported and enlarged}.

The results shown by Figure \ref{fig:moving-mass-labeling-results} closely resemble those of Figure \ref{fig:fern-labeling-results}, such that the corresponding comments apply likewise. We just point out again the following: Looking at the optimal weight patches in terms of their deviation from uniform weights, as depicted by the rightmost panel in the top row of Figure \ref{fig:moving-mass-labeling-results}, it is both interesting and not too difficult to understand -- after convergence and informally by visual inspection -- how these weights encode this particular `label transport'. However, predicting these weights and certifying their optimality \textit{beforehand}, seems to be an infeasible task. For example, it is hard to predict that the creation of intermediate locations where assignment mass temporarily accumulates (clearly visible in Fig.~\ref{fig:moving-mass-labeling-results}), effectively optimizes the constrained functional \eqref{eq:parameter-problem-modified-laf}. \textit{Learning} these weights, on the other hand, just requires to apply our approach.

\newlength{\ImgSizeSupervised}
\setlength{\ImgSizeSupervised}{0.11\textwidth}
\begin{figure}[h!]
\begin{center}
    \begin{tabularx}{0.75\linewidth}{c c c c}
	\fbox{\includegraphics[width=1.3\ImgSizeSupervised]{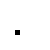}} &
	\fbox{\includegraphics[width=1.3\ImgSizeSupervised]{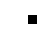}} &
	\fbox{\includegraphics[width=1.3\ImgSizeSupervised]{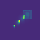}}&
	\includegraphics[height=1.3\ImgSizeSupervised]{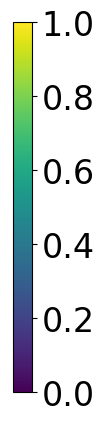}\\
\parbox{1.5\ImgSizeSupervised}{\centering \textcolor{black}{Input image}} &
\parbox{1.5\ImgSizeSupervised}{\centering \textcolor{black}{Target labeling}} &
\parbox{1.5\ImgSizeSupervised}{\centering \textcolor{black}{difference of adaptive\\[-1mm] to uniform \textit{weights}}}
    \end{tabularx}
    \end{center}
    
\begin{center}
    \begin{tabularx}{1\linewidth}{c c c c c c c}
        \rotatebox{90}{{\hspace{1mm} \textbf{Linear}}} &
        \fbox{\includegraphics[width=\ImgSizeSupervised]{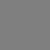}}&
	\fbox{\includegraphics[width=\ImgSizeSupervised]{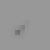}} &
	\fbox{\includegraphics[width=\ImgSizeSupervised]{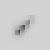}} &
	\fbox{\includegraphics[width=\ImgSizeSupervised]{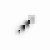}} &
	\fbox{\includegraphics[width=\ImgSizeSupervised]{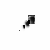}} &
	\fbox{\includegraphics[width=\ImgSizeSupervised]{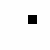}} \\
	        \rotatebox{90}{{\hspace{1mm} \textbf{Nonlinear}}} &
        \fbox{\includegraphics[width=\ImgSizeSupervised]{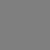}}&
	\fbox{\includegraphics[width=\ImgSizeSupervised]{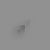}} &
	\fbox{\includegraphics[width=\ImgSizeSupervised]{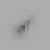}} &
	\fbox{\includegraphics[width=\ImgSizeSupervised]{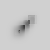}} &
	\fbox{\includegraphics[width=\ImgSizeSupervised]{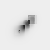}} &
	\fbox{\includegraphics[width=\ImgSizeSupervised]{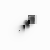}} \\
	&
\parbox{\ImgSizeSupervised}{\centering\scriptsize $t=0$} & 
\parbox{\ImgSizeSupervised}{\centering\scriptsize $t=2$} & 
\parbox{\ImgSizeSupervised}{\centering\scriptsize $t=3.5$} &
\parbox{\ImgSizeSupervised}{\centering\scriptsize $t=5$} &
\parbox{\ImgSizeSupervised}{\centering\scriptsize $t=6.5$} &
\parbox{\ImgSizeSupervised}{\centering\scriptsize $T=8.5$}
    \end{tabularx}
\end{center}
\caption{\textbf{Transporting and Enlarging Label Assignments.} See Fig.~\ref{fig:fern-labeling-results} for the set-up. \textsc{top row:} Label locations of the input data \textit{do not} form a subset of the target locations. Thus, `mass' of label assignments has to be both transported and enlarged. Rightmost panel: Distance of the \textit{optimal} weight patch from uniform weights, for every pixel.   
\textsc{middle row:} Applying our approach to \eqref{eq:parameter-problem-modified-laf} effectively solves the problem. \textsc{bottom row:} 
Inserting the optimal weights that are computed using the \textit{linear} assignment flow into the \textit{nonlinear} assignment flow gives a similar result and underlines the good approximation property of the linear assignment flow. It is interesting to obverse that computing the Riemannian gradient flow on the parameter manifold entails `intermediate locations' where assignment mass accumulates temporarily. This underlines the necessity of learning, since it seems hard to predict such an \textit{optimal} regularization strategy beforehand.
}
\label{fig:moving-mass-labeling-results}
\end{figure}

\subsubsection{Parameter Learning vs.~Optimal Control}\label{sec:Parameter-Learning-vs-OC}
Figure \ref{fig:exceeding-the-end-time-labeling-results} illustrates limitations of our parameter learning approach. In this experiment, we simply \textit{exceeded} the time horizon 
in order to inspect labelings induced by the linear assignment flow \textit{after} the point of time $T$, that was used for determining optimal weights in the training phase. Starting with $T$, Figure \ref{fig:exceeding-the-end-time-labeling-results} shows these labelings for both experiments corresponding to Figures \ref{fig:fern-labeling-results} and \ref{fig:moving-mass-labeling-results}.

Unlike the fern pattern (top row) where the initial label locations formed a subset of the target locations, the `moving mass pattern' (bottom row) is \textit{unsteady} in the following quite natural sense: the linear assignment flow simply continues transporting mass beyond time $T$. As a result, assignments to the white label are transported to locations of the black target pattern. Hence, the target pattern is first created up to time $T$ and destroyed afterwards. 

This behaviour is not really a limitation, but a consequence of merely learning \textit{constant} weight parameters. Due to the formulation of the optimization problem \eqref{eq:parameter-problem-modified-laf}, optimal weights not only encode the `knowledge' how to steer the assignment flow in order to solve the problem, but also the \textit{time period} after which the task has to be completed. Fixing this issue requires a higher-level of adaptivity: weight \textit{functions} depending on time and the current state of assignments would have to be estimated, that may be adjusted online through feedback in order to \textit{control} the assignment flow in a more flexible way.

\begin{figure}[h!]
\begin{center}
    \begin{tabularx}{0.8\linewidth}{c c c c c}	
    \fbox{\includegraphics[width=\ImgSizeFeather]{./feather_frame50}} &
	\fbox{\includegraphics[width=\ImgSizeFeather]{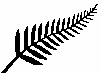}} &
	\fbox{\includegraphics[width=\ImgSizeFeather]{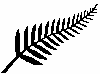}} &
	\fbox{\includegraphics[width=\ImgSizeFeather]{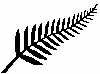}} &
	\fbox{\includegraphics[width=\ImgSizeFeather]{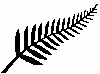}} \\
	\parbox{\ImgSizeSupervised}{\centering\scriptsize $t=T$} & 
\parbox{\ImgSizeSupervised}{\centering\scriptsize $t=T+0.5$} & 
\parbox{\ImgSizeSupervised}{\centering\scriptsize $t=T+1$} &
\parbox{\ImgSizeSupervised}{\centering\scriptsize $t=T+1.5$} &
\parbox{\ImgSizeSupervised}{\centering\scriptsize $t=T+2$}\\[2mm]
        \fbox{\includegraphics[width=\ImgSizeSupervised]{./moving_mass_frame17}}&
	\fbox{\includegraphics[width=\ImgSizeSupervised]{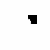}} &
	\fbox{\includegraphics[width=\ImgSizeSupervised]{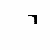}} &
	\fbox{\includegraphics[width=\ImgSizeSupervised]{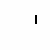}} &
	\fbox{\includegraphics[width=\ImgSizeSupervised]{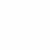}} \\
\parbox{\ImgSizeSupervised}{\centering\scriptsize $t=T$} & 
\parbox{\ImgSizeSupervised}{\centering\scriptsize $t=T+1.5$} & 
\parbox{\ImgSizeSupervised}{\centering\scriptsize $t=T+3$} &
\parbox{\ImgSizeSupervised}{\centering\scriptsize $t=T+4$} &
\parbox{\ImgSizeSupervised}{\centering\scriptsize $t=T+5$}
    \end{tabularx}
\end{center}
\caption{\textbf{Parameter Learning vs.~Optimal Control.} The plots show label assignments by computing the assignment flow \textit{beyond} the final point of time $T$ used during training, for the experiments corresponding to Figures \ref{fig:fern-labeling-results} and \ref{fig:moving-mass-labeling-results}. Unlike the pattern completion experiment (top row) where few locations of initial \textcolor{black}{data formed a subset of the target pattern}, the target pattern (bottom row, at time $T$) of the moving-mass experiment is \textit{unsteady} in the following sense: at time $T$, the flow continues to transport mass which eventually erases the target pattern with assignments of the white background label. 
The reason is that \textit{constant}
parameters are only learned that not only encode the `knowledge' how to steer the flow to the target pattern but also the time period $[0,T]$ for accomplishing this task. In order to remedy this limitation, weight \textit{functions} depending on the assignments (state of the assignment flow) would have to be estimated by applying techniques of optimal control.
}
\label{fig:exceeding-the-end-time-labeling-results}
\end{figure}

\section{Conclusion}
\label{sec:Conclusion}

We introduced a parameter learning approach for image labeling based on the assignment flow. During the training phase, weights for geometric averaging of label assignments are estimated from ground truth labelings, in order to steer the flow to prescribed labelings. Using the linearized assignment flow, we showed that, by using a class of symplectic partitioned Runge-Kutta methods, this task can be accomplished by numerically integrating the adjoint system in a consistent way. Consistent means that discretization and differentiation for the training problem \textit{commute}. An additional convenient property of our approach is that the parameter manifold has mathematically the structure of an assignment manifold, such that Riemannian gradient descent can be used for effectively solving the training problem.

The output of the training phase is a database containing features extracted from training data, together with the respective optimal weights. In order to complete the parameter learning task, a mapping has to be specified that predicts optimal weights for novel unseen data. We solved this task simply by nearest-neighbor prediction after partitioning the database using $k$-means clustering and geometric averaging of the weights, separately for each cluster. We evaluated this approach \textcolor{blue}{for a binary label scenario consisting of letters and} a 3-label scenario involving line structures where just using uniform weights inevitably fails \textcolor{blue}{in both cases}. We additionally conducted experiments that highlight the model expressiveness of the assignment flow and also limitations caused by merely learning \textit{constant} parameters.

\vspace{0.25cm}
Our main insights include the following. Regarding numerical optimization for parameter learning in connection with image labeling, our approach is more satisfying than working with discrete graphical models, where parameter learning requires to evaluate the partition function, which is a much more involved task when working with cyclic grid graphs. This latter problem of computational \textit{statistics} shows up in our scenario in similar form as the problem to design the prediction map from features to weight parameters. A key difference of these two scenarios is that by restricting the scope to statistical predictions at a \textit{local} scale, i.e.~only within small windows, the prediction task becomes \textit{manageable}, since regarding numerical optimization, no further approximations are involved at all. 

Regarding future work, we mention two directions. The natural way for broading the scope of the prediction map and the class of images that the assignment flow can represent, is the composition of two or several assignment flows in a hierarchical fashion. This puts our work closer to current mainstream research on deep networks, whose parametrizations and internal representations are not fully understood, however. We hope that using the assignment flow can help to  understand hierarchical architectures better.

The second line of research concerns the learning of weight \textit{functions}, rather than constant parameters, as motivated in Section \ref{sec:Parameter-Learning-vs-OC}, since this would also enhance model expressiveness and adaptivity considerably. A key problem then is to clarify the role of these functions and the choice of an appropriate time scale, as part of an hierarchical composition of assignment flows. 

\begin{appendix}
\label{sec:Appendix}

\section{Proofs of Section~2}\label{appendix-proofs-section2}

\subsection{Proof of Theorem~\ref{thm:sensitivity-continuous}}\label{appendix-proof-sensitivity-continuous}

A proof can be found, e.g., in \cite{cao_adjoint_2003}. However, in order to make this paper self-contained, we include a proof here.

\begin{proof}
Setting up the Lagrangian
\begin{equation}
\mc{L}(x, p, \lambda) = \mc{C}(x(T)) - \int_0^T \big\la \lambda, F(\dot x, x, p, t) \big\ra dt
\end{equation}
with multiplier $\lambda(t)$ and $F(\dot x, x, p, t) := \dot x - f(x, p, t) \equiv 0$, we get with $\Phi(p) = \mc{C}(x(T))$ from \eqref{eq:def-Phi}
\begin{equation}\label{eq:general-formula-sensitivity}
\begin{split}
\Egrad \Phi = \Egrad_p \mc{L} &= d_p x(T)^{\T} \Egrad \mc{C}\big(x(T)\big)
-
\int_0^T \Big(
d_{\dot x} F d_p\dot x + d_x F d_p x
+ d_p F
\Big)^{\T} \lambda \; dt,
\end{split}
\end{equation}
where integration applies component-wise.
By using $d_{\dot x} F = I$, where $I$ denotes the identity matrix, we partially integrate the first term under the integral,
\begin{equation}
\int_0^T d_p \dot x^{\T} \lambda \; dt = d_p x^{\T} \lambda\bigg|_{t=0}^{T} - \int_0^T d_p x^{\T} \dot\lambda \; dt.
\end{equation}
We further obtain with $d_p F = -d_p f$ and $d_x F = -d_x f$
\begin{subequations}
\begin{align}
\begin{split}
\Egrad \Phi
&= d_p x(T)^{\T} \Egrad \mc{C}(x(T)) - d_p x^{\T} \lambda \bigg|_{t=0}^{T} + \int_0^T d_p x^{\T} \dot \lambda \; dt + \int_0^T \Big(d_x f d_p x + d_p f \Big)^{\T} \lambda \; dt
\end{split} \\
\begin{split}
&= d_p x(T)^{\T} \Egrad \mc{C}(x(T)) - d_p x(T)^\T \lambda(T) + d_p x(0)^\T \lambda(0) + \int_0^T d_p x^\T \dot \lambda +  d_p x^\T d_x f^\T \lambda + d_p f^{\T} \lambda\; dt.
\end{split}
\intertext{We consider systems where the initial value $x_0$ is independent of the parameter $p$, i.e. $d_p x(0) = 0$. Additionally factoring out the unknown Jacobian $d_p x$, we obtain}
\begin{split}
&= d_p x(T)^{\T} \Big (\Egrad \mc{C}(x(T)) - \lambda(T) \Big ) + \int_0^T d_p x^{\T} \Big(\dot \lambda + d_x f^\T \lambda \Big ) + d_p f^{\T} \lambda\; dt.
\end{split}
\end{align}
\end{subequations}
Now, by choosing $\lambda(t)$ such that conditions \eqref{eq:two-point-boundary-b} are fulfilled, i.e.
\begin{equation*}
\dot \lambda(t) = - d_x f^\T \lambda(t), \qquad \lambda(T) = \Egrad_x \mc{C}(x(T)),
\end{equation*} 
we finally obtain 
\begin{equation}
  \Egrad \Phi = \int_0^T d_p f^{\T} \lambda(t) \; dt.
\end{equation}
\end{proof}

\subsection{Proof of Theorem~\ref{theorem:discrete-sensitivity}}\label{appendix:theorem-discrete-sensitivity}
For the proof of this theorem we follow the suggested outline of \cite{Sanz-Serna:2016aa}: State the Lagrangian of the nonlinear problem \eqref{eq:resulting-NLP-general} and apply the following lemma, which is a slightly different version of Lemma 3.5 in \cite{Sanz-Serna:2016aa}.

\begin{lemma}\label{lemma-gradient-lagrangian}
Suppose that the mapping $\phi \colon \R^{n_p \times d'} \to \R^{d'}$ is such that the Jacobian matrix $d_{\gamma} \phi$ is invertible at a point $(p_0, \gamma_0) \in \R^{n_p} \times \R^{d'}$, that is in the neighborhood of $p_0$, the equation $\phi(p, \gamma) = 0$ defines $\gamma$ as a function of $p$. For some given function 
$\mathcal{C} \colon \R^{n_p  \times d'} \to \R$ consider the induced function of the form $\Phi \colon \R^{n_p}  \to \R$, defined by $\Phi(p) := \mc{C}(p, \gamma(p))$.
We introduce the Lagrangian 
\begin{equation}
\mathcal{L}(p, \gamma, \lambda) = \mc{C}(p, \gamma) + \la \phi(p, \gamma), \lambda \ra.
\end{equation} 
Then, the Euclidean gradient of $\Phi$ with respect to $p$ at $p_0$ is given by
\begin{equation}\label{eq:grad-objective-grad-lagrange}
\Egrad \Phi(p_0) = \Egrad_{p} \mc{L}(p_0, \gamma_0, \lambda_0),
\end{equation}
where the vectors $\gamma_0 = \gamma(p_0) \in \R^{d'}$ and $\lambda_0 \in \R^{d'}$ are uniquely determined by
\begin{subequations}
\begin{align}
0 &= \Egrad_{\lambda} \mc{L}(p_0, \gamma_0, \lambda_0) = \phi(p_0, \gamma_0), \label{eq:Runge-Kutta-lagrangian-a}\\
\begin{split}
0 &= \Egrad_{\gamma} \mc{L}(p_0, \gamma_0, \lambda_0)\quad \Longleftrightarrow \quad \Egrad_{\gamma} \mc{C}(p_0, \gamma_0) = - d_{\gamma} \phi(p_0, \gamma_0)^\T \lambda_0. \label{eq:Runge-Kutta-lagrangian-b}
\end{split}
\end{align}
\end{subequations}
\begin{proof}
Since we evaluate all occurring functions and their derivatives at the same points $p_0$, $\gamma_0$ and $\lambda_0$, we drop them as arguments in the following, to simplify notation. 

\noindent
(i) Equation \eqref{eq:Runge-Kutta-lagrangian-a} directly follows by differentiating $\mc{L}$ with respect to $\lambda$ at $(p_0, \gamma_0, \lambda_0)$. 

\noindent
(ii) Equation \eqref{eq:Runge-Kutta-lagrangian-b} is immediately obtained by differentiating $\mc{L}$ with respect to $\gamma$ at $(p_0, \gamma_0, \lambda_0)$. 
Since $d_\gamma \phi$ is invertible at $(p_0, \gamma_0)$, the resulting linear system uniquely determines the vector $\lambda_0$. 

\noindent
(iii) Next, we show that this $\lambda_0$ also satisfies the first equation \eqref{eq:grad-objective-grad-lagrange}. By differentiating $\phi(p, \gamma) = 0$ with respect to $p$ at $(p_0, \gamma_0)$, we obtain
\begin{equation}\label{eq:implicit-function-theo}
d_{\gamma} \phi d_p \gamma_0 + d_{p} \phi = 0 \; \stackrel{d_{\gamma} \phi \;\text{is invertible}}{\Longleftrightarrow} \; d_p \gamma_0 = - (d_{\gamma} \phi)^{-1} d_{p} \phi.
\end{equation}
We will make use of this identity for $d_p \gamma_0$ in the following. Differentiating $\Phi$ with respect to $p$ at $p_0$ and by the chain rule, we obtain
\begin{subequations}
\begin{align}
\Egrad \Phi &\stackrel{\phantom{\eqref{eq:implicit-function-theo}}}{=} \Egrad_p \mc{C} + d_p\gamma_0^\T \Egrad_\gamma \mc{C} \stackrel{\eqref{eq:Runge-Kutta-lagrangian-b}}{=}\Egrad_{p} \mc{C} - d_{p} \gamma_0^\T d_{\gamma} \phi^\T \lambda_0\\
&\stackrel{{\eqref{eq:implicit-function-theo}}}{=}\Egrad_{p} \mc{C} + \big ((d_{\gamma} \phi)^{-1} d_{p} \phi \big)^\T d_{\gamma} \phi^\T \lambda_0 = \Egrad_{p} \mc{C} + d_{p} \phi^\T\lambda_0 \\
&\stackrel{\phantom{\eqref{eq:implicit-function-theo}}}{=} \Egrad \mathcal{L},
\end{align}
\end{subequations}
which shows \eqref{eq:grad-objective-grad-lagrange}. 
\end{proof}
\end{lemma}

\begin{proof}[Proof of Theorem~\ref{theorem:discrete-sensitivity}]
We begin by stating the Lagrangian of problem \eqref{eq:resulting-NLP-general}
\begin{align}\label{eq:def-RK-Lagrangian}
\mc{L}(x,p, \lambda) &= \mc{C}\big( x_N \big ) - \lambda_0^\T (x_0 - x(0))- \sum_{n=0}^{N-1} \lambda_{n+1}^\T \Big [x_{n+1} - x_n - h_n\sum_{i=1}^s b_i k_{n,i}\Big ]\\
& \quad - \sum_{n=0}^{N-1} h_n \sum_{i=1}^s b_i \Lambda_{n,i}^\T \Big [ k_{n,i} - f(X_{n,i}, p, t_n + c_i h_n)\Big ].\nonumber
\end{align}
In order to apply Lemma \ref{lemma-gradient-lagrangian}, we explain which role the variables $\gamma, \lambda, \phi$
 play in this situation:

\begin{enumerate}
\item \textit{Intermediate stages:} The vector $\gamma$ represents all intermediate stages related to the evaluation of the function $\Phi(p) = \mc{C}(x_N(p))$, i.e. all intermediate values $x_i$ and 
stages $k_i$ of the Runge--Kutta method. These variables are stacked and arranged as follows
\begin{equation}\label{eq:intermediate-stages}
\begin{split}
&\mathbf{\gamma} = \begin{bmatrix}
x_0\\
\gamma_0\\
\gamma_1\\
\vdots\\
\gamma_{N-1}
\end{bmatrix} \in \R^{d'}, \quad \gamma_n = \begin{bmatrix}
{k}_{n} \\ x_{n+1}
\end{bmatrix} \in \R^{(s+1)n_x}, \quad \text{and} \quad {k}_{n} = \begin{bmatrix}
k_{n,1} \\ \vdots \\ k_{n,s}
\end{bmatrix}\in \R^{sn_x}. 
\end{split}
\end{equation}

\item \textit{Lagrange multiplier:} The vector $\lambda$ contains all Lagrange multipliers in \eqref{eq:def-RK-Lagrangian}
 belonging to the constraints \eqref{eq:NLP-general-a}-\eqref{eq:NLP-general-c}. The multipliers are stacked and arranged as follows
\begin{equation}\label{eq:def-multiplier}
\lambda = \begin{bmatrix}
-\lambda_0\\
-\Lambda_0\\
\vdots\\
-\lambda_{N-1}\\
-\Lambda_{N-1}\\
-\lambda_{N}
\end{bmatrix} \in \R^{d'},\quad\Lambda_n = \begin{bmatrix}
h_n b_1\Lambda_{n,1}\\
\vdots\\
h_n b_s\Lambda_{n,s}
\end{bmatrix} \in \R^{sn_x}.
\end{equation}

\item \textit{Intermediate mappings:} Analogously, the vector $\phi$ contains all intermediate 
mappings $\phi_n$ of the computation of $\Phi(p) = \mc{C}(x_N(p))$ with $n = 1, \dots, N-1$. In our situation, $\phi$ is the concatenation of the \textit{forward} 
Runge--Kutta evaluation, which we express using the Kronecker-product as
\begin{equation}\label{eq:intermediate-mappings}
\phi = \begin{bmatrix}
x_0 - x(0)\\
{\Psi}_1 \\
{\Psi}_2 \\
\vdots\\
{\Psi}_{N-1}
\end{bmatrix}\in \R^{d'},\quad
{\Psi}_n = \begin{bmatrix}
k_n - F_n(X_n, p)\\
x_{n+1} - x_n-h_n (b^{\T} \otimes I_{n_x}) k_n
\end{bmatrix}= \begin{bmatrix}
\Psi_{n,1} \\
\Psi_{n,2}
\end{bmatrix}\in \R^{(s+1)n_x}, 
\end{equation}
where $\Psi_{n,1} \in \R^{sn_x}$ and $\Psi_{n,2} \in \R^{n_x}$, as well as
\begin{equation}
\begin{split}
&F_n(X_n, p) = \begin{bmatrix}
f(X_{n,1}, p, t_n + c_1 h_n)\\ \vdots \\f(X_{n,s}, p, t_n + c_s h_n)
\end{bmatrix}, \quad X_n =\eins_s \otimes x_n+h_n(A \otimes I_{n_x}) k_n = \begin{bmatrix} X_{n,1}\\ \vdots\\X_{n,s}\end{bmatrix}.
\end{split}
\end{equation}
\end{enumerate} 
We proceed by computing the Jacobian $d_{\gamma} \phi$. Note that the intermediate variables $\gamma_n$ \eqref{eq:intermediate-stages} are only contained in the intermediate 
mappings $\Psi_n$ \eqref{eq:intermediate-mappings}, which results in a sparse block structure of the overall Jacobian $d_{\gamma} \phi$.
 
 \begin{enumerate}
\item \textit{Small block matrices:} Each small block matrix represents the derivative of the $n$-th iteration step $\Psi_n$ and is given by 
\begin{equation}\label{eq:local-Jacobian}
\begin{split}
d_{(x_n, k_n, x_{n+1})} \Psi_n &=
\begin{bmatrix}
d_{x_n}\Psi_{n,1} & d_{k_n}\Psi_{n,1} & d_{x_{n+1}}\Psi_{n,1}\\
d_{x_n} \Psi_{n,2}& d_{k_n}\Psi_{n,2} & d_{x_{n+1}}\Psi_{n,2}
\end{bmatrix}=
\begin{bmatrix}
{D}_n & {A}_n & \\
- {I}_{n_x} & {B}_n^\T & {I}_{n_x}
\end{bmatrix},
\end{split}
\end{equation}
with
\begin{subequations}
\begin{align}
{A}_n &= {I}_{sn_x} - h_n d_x F_n(X_n, p) ({A}\otimes I_{n_x}),\label{eq:simplified-notation-a}\\
{B}_n^\T &= - h_n {b}^\T \otimes {I}_{n_x}, \\
{D}_n &= -d_x F_n(X_n, p) (\eins_s \otimes I_{n_x}),\label{eq:simplified-notation-b}
\end{align}
\end{subequations}
where ${A}$ and ${b}$ are the Runge-Kutta coefficients given by the above tableau of Fig.~\ref{fig:butcher}. 

\item \textit{Sparse block structure:} The overall Jacobian $d_{\gamma} \phi$ consists of $N-1$ blocks (one for each iteration) of the form \eqref{eq:local-Jacobian} and is given by
\begin{equation}\label{eq:overall-Jacobian}
d_{\gamma} {\phi} = \begin{bmatrix}
{I}_{n_x}&&&&&\\
{D}_1&{A}_1 &&&&\\
- {I}_{n_x}&{B}_1^\T &{I}_{n_x} & &&& & \\
&&{D}_2 & {A}_2 & &&&\\
&&- {I}_{n_x} & {B}_2^\T & {I}_{n_x} &&&\\
&&&\ddots&\ddots&\ddots&&\\
&&&&&{D}_{N-1} & {A}_{N-1} & \\
&&&&&-{I}_{n_x} & {B}_{N-1}^\T & {I}_{n_x}
\end{bmatrix}.
\end{equation}

\item \textit{Invertibility of $d_{\gamma} {\phi}$:} A matrix $M$ is invertible if $\det M \ne 0$. Since the matrix $d_{\gamma} {\phi}$ is lower block diagonal, its determinant is 
given by
\begin{equation}
\det d_{\gamma} {\phi} = \det A_1 \cdot \ldots \cdot \det A_{N-1}
\end{equation}
Thus, we only need to show that $\det A_n \ne 0$ for all $n = 1, \dots, N-1$. Equation \eqref{eq:simplified-notation-a} reads in a more compact form
\begin{equation}\label{eq:def-M}
{A}_n = ({I}_{sn_x} - h_n M), \text{ with } M:= d_x F_n(X_n, p) ({A}\otimes I_{n_x}).
\end{equation}
We show $\det A_n \ne 0$ by using the equivalent statement  $\ker(A_n) = \{0\}$.
Now, let $x \in \R ^{sn_x} \backslash\{0\}$, then
\begin{equation}\label{eq:trivial-kern}
\|x\| \leq\|x-h_n M x\|+\|h_n M x\| \leq\|A_n x\|+\|h_n M\|\|x\|.
\end{equation}
By using the row-sum norm $\|\cdot \|_\infty$, we have 
\begin{align}
\|h_n M\|_\infty &\stackrel{\eqref{eq:def-M}}{=} h_n\| d_x F_n(X_n, p) ({A}\otimes I_{n_x})\|_\infty\stackrel{\phantom{\eqref{eq:def-M}}}{\le} h_n\| d_x F_n(X_n, p) \|_\infty  \|({A}\otimes I_{n_x})\|_\infty \nonumber\\
&\stackrel{\phantom{\eqref{eq:def-M}}}{<} h_n L \max_{i=1, \dots, s}\sum_{j=1}^s |a_{ij}| \stackrel{\eqref{eq:step-size-condition}}{<} 1,\label{eq:inequality}
\end{align}
where $L$ denotes the Lipschitz constant of $f$ and the step-size $h_n$ satisfies the assumption \eqref{eq:step-size-condition}. Substituting \eqref{eq:inequality} into \eqref{eq:trivial-kern} gives
\begin{equation}
\|x\| < \|A_n x\|+\|x\| iff 0 < \|A_n x\| \iff x \not \in \ker(A_n).
\end{equation}
Since, the kernel of $A_n$ is trivial, $A_n$ is invertible and consequently the overall Jacobian $d_\gamma \phi$ as well.  
\end{enumerate}
Now we are in a position to apply Lemma~\ref{lemma-gradient-lagrangian}. More precisely, \eqref{eq:Runge-Kutta-lagrangian-b} tells us that the vector $\lambda$ is uniquely determined by the linear system $d_{\gamma} \phi^\T \lambda = - \Egrad_{\gamma} \mc{C}$.
In our situation, this system is given by
\begin{align}
&{\begin{bmatrixT}
{I}_{n_x}&&&&&\\
{D}_1^\T&{A}_1^\T &&&&\\
- {I}_{n_x}&{B}_1 &{I}_{n_x} & &&& & \\
&&{D}_2^\T & {A}_2^\T & &&&\\
&&- {I}_{n_x} & {B}_2 & {I}_{n_x} &&&\\
&&&\ddots&\ddots&\ddots&&\\
&&&&&{D}_{N-1}^\T & {A}_{N-1}^\T & \\
&&&&&-{I}_{n_x} & {B}_{N-1} & {I}_{n_x}
\end{bmatrixT}} \begin{bmatrix}
-\lambda_0\\
-\Lambda_0\\
-\lambda_1\\
-\Lambda_1\\
\vdots\\
-\lambda_{N-1}\\
-\Lambda_{N-1}\\
-\lambda_{N}
\end{bmatrix}
&= -\begin{bmatrix} 0\\ 0\\ 0\\ 0\\ \vdots \\ 0 \\ 0 \\ \Egrad_x \mc{C}(x_N)
\end{bmatrix}.\label{eq:linear-system-lambda}
\end{align}
We determine the exact identity of $\lambda$ by backward substitution:\\
\noindent
1. From the last row of \eqref{eq:linear-system-lambda}, we immediately obtain 
\begin{equation}
\lambda_{N} = \Egrad_x \mc{C}(x_N).
\end{equation}

\noindent
2. Next, we prove equation \eqref{eq:discrete-adjoints-RK-a}. For each $n = 0, \dots, N-1$, we obtain
\begin{align}
0 &= 
\begin{bmatrix}
 {I}_{n_x} & {D}_n^\T &  -{I}_{n_x}
\end{bmatrix}
\begin{bmatrix}
\lambda_n\\
\Lambda_n\\
\lambda_{n+1}
\end{bmatrix} = \lambda_n + {D}_n^\T \Lambda_n - \lambda_{n+1}\nonumber\\
&= \lambda_n - (d_x F_n(X_n, p) (\eins_s \otimes I_{{n_x}}))^\T \begin{bmatrix}
h_n b_1\Lambda_{n,1}\\
\vdots\\
h_n b_s\Lambda_{n,s}
\end{bmatrix}- \lambda_{n+1}\nonumber\\
&= \lambda_n - (\eins_s^\T \otimes I_{n_x}) d_x F_n(X_n, p)^\T \begin{bmatrix}
h_n b_1\Lambda_{n,1}\\
\vdots\\
h_n b_s\Lambda_{n,s}
\end{bmatrix} - \lambda_{n+1}\nonumber\\
&= \lambda_n - h_n \sum_{i=1}^s b_i d_x f(X_{n,i}, p, t_n + c_i h_n)^\T \Lambda_{n,i}- \lambda_{n+1}\nonumber.\\
&\lambda_{n+1} = \lambda_n + h_n \sum_{i=1}^s b_i \ell_{n,i}, \label{eq:discrete-adjoint-step-proof}
\end{align}
with $\ell_{n,i} = - d_x f(X_{n,i}, p, t_n + c_i h_n)^\T \Lambda_{n,i}$.\\

\noindent
3. The last equation \eqref{eq:discrete-adjoints-RK-c} follows by
\begin{align*}
0 &= \begin{bmatrix} 0 & {A}_n^\T & {B}_n \end{bmatrix} \begin{bmatrix}\lambda_n\\ \Lambda_n\\ \lambda_{n+1} \end{bmatrix}\\
&= {A}_n^\T \Lambda_n + {B}_n \lambda_{n+1}\\
&= \left ( {I}_{s{n_x}} - h_n d_x F_n(X_n, p) ({A}\otimes I_{n_x})\right )^\T \Lambda_n - h_n ({b} \otimes {I}_{n_x} )\lambda_{n+1}\\
&= \left ( {I}_{s{n_x}} - h_n (A^\T \otimes I_{n_x})d_x F_n(X_n, p)^\T \right ) \Lambda_n - h_n (b \otimes {I}_{n_x})\lambda_{n+1}.
\end{align*}
In the following, we consider the $i$-th entry of the previous equation, i.e. $h_nb_i\Lambda_{n,i}$ of $\Lambda_n$ with $i = 1, \dots, s$.
\begin{align*}
&0 = h_n b_i \Lambda_{n,i} - h_n^2 \sum_{j=1}^s a_{ji}b_j \partial_x f(X_{n,j}, p, t_n + c_j h_n)^\T \Lambda_{n,j} - h_n b_i \lambda_{n+1}\\
&\Lambda_{n,i} = \lambda_{n+1} + h_n \sum_{j=1}^s \frac{a_{ji}b_j}{b_i} d_x f(X_{n,j}, p, t_n + c_j h_n)^\T \Lambda_{n,j}\\
 &\stackrel{\eqref{eq:discrete-adjoint-step-proof}}{=} \lambda_n + h_n \sum_{i=1}^s b_i \ell_{n,i} - h_n \sum_{j=1}^s \frac{a_{ji}b_j}{b_i}\ell_{n,j}\\
 &\stackrel{\phantom{\eqref{eq:discrete-adjoint-step-proof}}}{=} \lambda_n + h_n \sum_{j=1}^s \left(b_j -\frac{a_{ji}b_j}{b_i}\right)\ell_{n,j},
\end{align*}
with $\ell_{n,j} = - d_x f(X_{n,j}, p, t_n + c_j h_n)^\T \Lambda_{n,j}$.

\vspace{2mm}
Finally, we show the formula of the gradient \eqref{eq:discrete-sensitivity-parameter}, which is given by \eqref{eq:grad-objective-grad-lagrange}
\begin{equation}
 \Egrad \Phi = \Egrad_{p} \mc{C} + d_{p} \phi^\T \lambda_0 \stackrel{\Egrad_{p} \mc{C}=0}{=} d_{p} \phi^\T \lambda_0.
\end{equation}
The Jacobian $d_{p} \phi^\T$ consists of the following building blocks: For the $n$-th iteration step $\Psi_n$ the \textit{local} Jacobian with respect to parameter $p$ reads
\begin{equation}\label{eq:local-Jacobian-wrt-parameter}
d_p \Psi_n =
\begin{bmatrix}
d_p \Psi_{n,1} \\
d_p \Psi_{n,2}
\end{bmatrix}=
\begin{bmatrix}
{\bar D}_n \\
0
\end{bmatrix},\quad 
{\bar D}_n = - d_p F_n(X_n, p) (\eins_s \otimes I_{n_p}).
\end{equation}
By concatenating $N-1$ of these blocks (one for each iteration $n = 1, \dots, N-1$) of \eqref{eq:local-Jacobian-wrt-parameter}, the overall Jacobian is given by
\begin{equation}
d_{p} \phi^\T =
\begin{bmatrix}
0&
{\bar D}_0^\T &
0&
{\bar D}_1^\T &
0&
\ldots &
0&
{\bar D}_{N-1}^\T&
0
\end{bmatrix}.
\end{equation}
Now, formula \eqref{eq:discrete-sensitivity-parameter} is explicitly given by
 \begin{align*}
 \Egrad \Phi &\stackrel{\phantom{\eqref{eq:def-multiplier}}}{=} d_{p} \phi^\T \lambda_0\\
 &= \begin{bmatrix}
0&
{\bar D}_0^\T &
0&
\dots &
0&
{\bar D}_{N-1}^\T &0
\end{bmatrix}\begin{bmatrix}
-\lambda_0\\
-\Lambda_0\\
-\lambda_1\\
\vdots\\
-\lambda_{N-1}\\
-\Lambda_{N-1}\\
-\lambda_{N}
\end{bmatrix}\stackrel{\phantom{\eqref{eq:def-multiplier}}}{=} -\sum_{n=0}^{N-1} {\bar D}_n^\T  \Lambda_{n}\\ 
&\stackrel{\phantom{\eqref{eq:def-multiplier}}}{=} \sum_{n=0}^{N-1} (d_p F_n(X_n, p) (\eins_s \otimes I_{n_p}))^\T  \Lambda_{n}\stackrel{\phantom{\eqref{eq:def-multiplier}}}{=} \sum_{n=0}^{N-1} ((\eins_s \otimes I_{n_p})) d_p F_n(X_n, p)^\T \Lambda_{n} \\ 
&\stackrel{\eqref{eq:def-multiplier}}{=} \sum_{n=0}^{N-1} h_{n} \sum_{i=1}^{s} b_i d_p f(X_{n,i}, p, t_n +c_ih_n) ^\T \Lambda_{n,i}.
\end{align*}
\end{proof}
\end{appendix}

\bibliographystyle{amsalpha}
\bibliography{Parameter-Estimation}

\end{document}